\documentclass[12pt]{amsart}
\usepackage{amssymb,amsmath}
\usepackage{verbatim}
\usepackage{hyperref}
\usepackage{graphicx}

\newtheorem{theorem}{Theorem} \theoremstyle{definition}

\newtheorem{lemma}[theorem]{Lemma}

\theoremstyle{remark}
\newtheorem{remark}{Remark}
\numberwithin{remark}{section}
\numberwithin{theorem}{section}
\numberwithin{equation}{section}
\numberwithin{definition}{section}

\newcommand{\tri}{\mathcal{T}}
\newcommand{\triH}{\tri_H}

\newcommand{\diam}{\operatorname*{diam}}
\newcommand{\kernel}{\operatorname*{kernel}}
\newcommand{\support}{\operatorname*{supp}}

\newcommand{\N}{\mathcal{N}}

\newcommand{\IH}{\mathcal{I}_H}

\newcommand{\Cint}{C_{\IH}}

\newcommand{\Vf}{V_{\operatorname*{f}}}

\newcommand{\VHFE}{V^{\text{FE}}_{\operatorname*{H}}}

\newcommand{\VhFE}{V^{\text{FE}}_{\operatorname*{h}}}
\newcommand{\VHLOD}{V^{\text{LOD}}_{\operatorname*{H}}}
\newcommand{\VHLODl}{V^{\text{LOD}}_{\operatorname*{H},\ell}}
\newcommand{\XHLOD}{X^{\text{LOD}}_{\operatorname*{H}}}
\newcommand{\XHLODl}{X^{\text{LOD}}_{\operatorname*{H},\ell}}

\newcommand{\Pf}{\mathcal{P}_{\operatorname*{f}}}
\newcommand{\Pc}{\mathcal{P}_{\operatorname*{c}}}

\renewcommand{\H}{\text{H}}

\title[GFEM for quadratic eigenvalue problems]{Generalized finite element methods for quadratic eigenvalue problems}

\author{Axel M\r{a}lqvist}
\thanks{A. M\r{a}lqvist is supported by the Swedish Research Council and the Swedish foundation of strategic research.}
\address{Department of Mathematics, Chalmers University of Technology and University of Gothenburg}
\author{Daniel Peterseim}
\thanks{D. Peterseim is supported by the Hausdorff Center for Mathematics Bonn and by Deutsche Forschungsgemeinschaft in the Priority Program 1748 "Reliable simulation techniques in solid mechanics. Development of non-standard discretization methods, mechanical and mathematical analysis" under the project
"Adaptive isogeometric modeling of propagating strong discontinuities in heterogeneous materials".}
\address{Insitute for Numerical Simulation, Rheinische Friedrich-Wilhelms-Universit{\"a}t Bonn}

\date{\today}



\begin{document}

\begin{abstract}
We consider a large-scale quadratic eigenvalue problem (QEP), formulated using P1 finite elements on a fine scale reference mesh. This model describes damped vibrations in a structural mechanical system. In particular we focus on problems with rapid material data variation, e.g., composite materials. We construct a low dimensional generalized finite element (GFE) space based on the localized orthogonal decomposition (LOD) technique. The construction involves the (parallel) solution of independent localized linear Poisson-type problems. The GFE space is then used to compress the large-scale algebraic QEP to a much smaller one with a similar modeling accuracy. The small scale QEP can then be solved by standard techniques at a significantly reduced computational cost. We prove convergence with rate for the proposed method and numerical experiments confirm our theoretical findings.
\end{abstract}

\maketitle

\section{Introduction}
Quadratic eigenvalue problems appear in various engineering disciplines. Often they are the result of finite element modeling rather than established partial differential equations. A classical example is a damped vibrating structure. For a spatially discretized structure we have,
\begin{equation}\label{e:model}
 Kz+\lambda Dz+\lambda^2 Mz=0,
\end{equation}
where $K$ is the finite element stiffness matrix, $D$ is the damping matrix, $M$ is the mass
matrix, $\lambda$ is the eigenvalue, and $z$ is the corresponding eigenvector. Assuming all matrices to be real there are $2n$ finite eigenvalues ($n$ being the dimension of the matrices) that are real or complex conjugate. Furthermore, if $z$ is a right eigenvector of $\lambda$ then $\bar{z}$ is a right eigenvector of $\bar{z}$.
If $D$ is also symmetric the left and right eigenvectors are equal. See \cite{TiMe01} for a more detailed discussion on how the spectrum and eigenspaces depend on properties of the matrices, and for an extensive overview of applications see \cite{TiMe01,MeVo04,BeHiMeSc13,HwLiMe03}.

We will use a specific finite element realization of the more general formulation presented in Equation~\eqref{e:model} as a starting point for this work. We seek an eigenvalue $\lambda\in\mathbb{C}$ and eigenfunction $u\in \VhFE$, where $\VhFE$ is a finite element space, such that,
\begin{equation*}
(\kappa \nabla u,\nabla v)+\lambda d(u,v)+\lambda^2 (u,v)=0,\quad\forall v\in\VhFE,
\end{equation*}
where the three bilinear forms in the left hand side corresponds to the three matrices $K$, $D$, and $M$ in Equation~\eqref{e:model}. In particular we are interested in problems where the coefficient $\kappa$ varies rapidly in space modeling, e.g., rapid material data variation. It is well known that in order to capture the correct  behavior of the solution, the finite element space $\VhFE$ has to be large enough to resolve the data variations \cite{BaOs99}. In this paper, we will not question the properness of this model. We assume that the parameter $h$ was carefully chosen in an earlier modeling or discretization step.

There are various models for the damping. The simplest one is proportional damping where $D=\alpha_0 K+\alpha_1 M$ with $\alpha_0,\alpha_1\in\mathbb{R}$ using the matrix notation. In this case the eigenmodes actually coincide with the eigenmodes of the linear generalized eigenvalue problem $Kx=\hat{\lambda} Mx$ but with different eigenvalues. If a structure is made up of different components (or materials) each of the parts may be modeled using proportional damping with different constants. The full structure will then have spatially varying damping parameters $\alpha_0$, $\alpha_1$ and the damping will therefore not be proportional. In this work we will treat a general abstract damping bilinear form $d$ but in the numerical experiments we will focus on mass and stiffness type damping with spatially varying parameters.

Numerical algorithms for solving quadratic eigenvalue problems are often based on linearization  \cite{MeVo04,HwLiMe03,BeHiMeSc13}. This technique results in a linear, possibly non-symmetric, generalized eigenvalue problem. In this approach, the size of the system gets doubled which, in the presence of multiscale features (since the initial finite element space has to be very large) is a drawback. Another drawback might be that the  underlying structure of the quadratic eigenvalue problem can be lost. There are also numerical schemes that can be applied directly to the quadratic eigenvalue problem without linearization, e.g., \cite{TiMe01}. In this paper, we will work with the linearized system.

While the literature in this context of numerical linear algebra is rich, results on error analysis and convergence of finite element approximations to quadratic eigenvalue problems are rather limited \cite{BeDuRoSo00}. To our best knowledge, the literature on GFEMs or multiscale methods for eigenvalue problems  is limited to \cite{MaPe13} which treats the linear case. Since linearization is a natural approach also when analyzing finite element approximations to quadratic eigenvalue problems, the literature on non-symmetric generalized eigenvalue problems is very relevant. Standard references include \cite{Ka76,BaOs91,Bo10} and for non-compact operators \cite{DeNaRa78a,DeNaRa78b}. These works gives a mathematical foundation  for the convergence analysis presented in this paper.

This paper neither aims to improve existing linear algebra techniques for solving quadratic eigenvalue problems nor to invent new ones. The aim is to reduce the size of the system before applying any of the standard solvers and, thereby, to speed-up the overall computation significantly. This will be achieved by constructing a computable low dimensional subspace of $\VhFE$, based on the framework of localized orthogonal decomposition (LOD) \cite{MaPe11,MaPe13}. The space captures the main features of the eigenspaces and, in particular, the effect of the rapid oscillations induced by the rough diffusion coefficient $\kappa$. Given this low dimensional space we can solve the quadratic eigenvalue problem at a greatly reduced computational cost while the accuracy is largely preserved. Under weak assumptions on the damping, the error analysis shows that the GFEM approximation is very accurate (in the sense of super convergence) when compared to the reference finite element solution, independent of the variations in the multiscale data. The proofs are based on the classical theory for non-symmetric eigenvalue problems presented in \cite{BaOs91}. Numerical examples confirm our theoretical findings.

The remaining part of this paper is structured as follows. In Section~\ref{s:fem} we present the model problem and its linearization. Section~\ref{s:lod} is devoted to the proposed numerical method. In Section~\ref{s:error}, we will derive a convergence result for the approximation. Section~\ref{s:num} shows numerical experiments and Section~\ref{s:conclusion} presents some final conclusions.

\section{Finite element spaces, linearization, and problem formulation}\label{s:fem}
In this section we introduce finite element spaces, formulate the discrete model problem, linearize, and finally arrive at a detailed problem formulation.

\subsection{Conforming finite element spaces}
 Let $\tri_h,\tri_H$ denote regular finite element meshes of a computational domain $\Omega\subset\mathbb{R}^d$, $d=1,2,3$, into closed simplices with mesh-size functions $0<h<H\in L^\infty(\Omega)$. If no confusion seems likely, we use $h$ and $H$ also to denote the maximal mesh sizes. The first-order fine and coarse conforming finite element spaces are
\begin{gather}\label{e:couranta}
\VhFE:=\{v\in V\;\vert\;\forall T\in\tri_h,v\vert_T \text{ is polynomial of degree}\leq 1\},\\
\VHFE:=\{v\in V\;\vert\;\forall T\in\tri_H,v\vert_T \text{ is polynomial of degree}\leq 1\}.
\end{gather}
We assume that $\VHFE\subset \VhFE$. By $\N_h$ and $\N_H$ we denote the set of interior vertices of the meshes. For every vertex $z$, let $\phi_z$ and $\varphi_z$ denote the corresponding nodal basis function to $\VHFE$ and $\VhFE$ respectively.

\begin{remark} While the nestedness of spaces $\VHFE\subset \VhFE$ is rather essential for our theory, the nestedness of the underlying meshes is not. The coarse finite element space $
\VHFE$ could be any subspace of $\VhFE$ that admits a local basis $\{\tilde{\phi}_{z} \in \VhFE : z\in \N_H \}$ with $\diam{\support\phi_z}\approx H$ and $\|\nabla^k \tilde{\phi}_{z}\|_{W^{k,\infty}(\Omega)} \lesssim H^{-k}$ ($k=0,1$), and possibly further conditions such as a partition of unity property; see \cite{HaMoPe15}. The method is then also applicable in cases where the resolution of characteristic microscopic geometric features of the model requires a highly unstructured fine mesh. A prototypical construction for this scenario can be found in  \cite[Section 6.2 and 7.3]{MaPe13}.
\end{remark}

\subsection{Quasi-interpolation}
 We will use a Cl\'ement-type interpolation operator to restrict the reference mesh functions to the coarser mesh $\IH: \VhFE\rightarrow \VHFE$. The interpolant is defined in the following way. Given $v\in \VhFE$, $\IH v := \sum_{z\in\N_H}(\IH v)(z)\phi_z$ defines a Cl\'ement interpolant with nodal values
 \begin{equation*}
 (\IH v)(z):=\frac{(v, \phi_z)_{L^2(\Omega)}}{(1,\phi_z)_{L^2(\Omega)}}\quad\text{for }z\in\N_H.
 \end{equation*}
 There exists a generic  constant $\Cint$ such that for all $v\in \H^1_0(\Omega)$ and for all $T\in\triH$ it holds
\begin{equation}\label{e:interr}
  H_T^{-1}\|v-\IH v\|_{L^{2}(T)}+\|\nabla(v-\IH v)\|_{L^{2}(T)}\leq \Cint \| \nabla v\|_{L^2(\omega_T)},
\end{equation}
where $\omega_T:=\cup\{t\in\triH\;\vert\;T\cap t\neq\emptyset\}$, see \cite{CV99} for a more detailed discussion. The constant $\Cint$ depends on the shape regularity of the finite element meshes but not on the mesh sizes.

\subsection{Model problem}
We can now formulate the model problem. Find $u\in \VhFE$ and $\lambda\in\mathbb{C}$ such that
\begin{equation}\label{e:model2}
(\kappa \nabla u,\nabla v)+\lambda d(u,v)+\lambda^2 (u,v)=0,\quad\forall v\in\VhFE.
\end{equation}
\paragraph{{\bf Assumption A}}
We assume,
\begin{equation*}
\kappa\in[\kappa_1,\kappa_2]\text{ and }d\text{ to be real and bounded,}
\end{equation*}
where $0<\kappa_1\leq \kappa_2<\infty$.

The corresponding matrix form is seen in Equation~\eqref{e:model}. These different formulations are related in the following way:
$u=\sum_{z\in\mathcal{N}_h}x_z\varphi_z$, $(\kappa\nabla \varphi_z,\nabla \varphi_w)=K_{z,w}$, $d(\varphi_z,\varphi_w)=D_{z,w}$, and $(\varphi_z,\varphi_w)=M_{z,w}$.

\subsection{Linearization}
Linearization is achieved by introducing two new variables $x_1=u$ and $x_2=\lambda u$. We let $X=\VhFE\times \VhFE$, with norm $$\|(x_1,x_2)\|^2_X:=\|\sqrt{\kappa}\nabla x_1\|_{L^2(\Omega)}^2+\|x_2\|_{L^2(\Omega)}^2:=\|\sqrt{\kappa}\nabla x_1\|^2+\|x_2\|^2,$$ for all $x=(x_1,x_2)\in X$, i.e., we drop the subscript $L^2(\omega)$ when the domain $\omega=\Omega$.
We  consider the weak form: find $x\in X$ and  $\lambda\in\mathbb{C}$ such that,
\begin{equation}
a(x,y)=\lambda \,b(x,y),
\end{equation}
for all $y\in X$ where the bilinear forms $a$ and  $b$ are defined as
\begin{gather}
a\left((x_1,x_2),(y_1,y_2)\right)=k( x_1, y_1)+( x_2,y_2)\\
b\left((x_1,x_2),(y_1,y_2)\right)=-d(x_1,y_1)-(x_2,y_1) +( x_1,y_2).
\end{gather}
Here several different choices are possible. The damping can be kept in the left hand side and the relation given by variations over $y_2$ can be done using other bilinear forms or expressed in matrix wording, by any invertible matrix. Our choice is common in the literature and fits our error analysis. For further discussion see, e.g., \cite{BeHiMeSc13}.

We note that with this choice $a$ is real, bounded and coercive, and $b$ is real and bounded. Under these assumptions there are unique linear bounded operators $A:X\rightarrow X$ and $A^*:X\rightarrow X$ satisfying,
\begin{gather}\label{e:A}
a(A x,y)=b(x,y),\quad \forall y\in X,\\
a(x,A^* y)=b(x,y),\quad \forall y\in X,
\end{gather}
see \cite{BaOs91}.

The analysis in this paper considers an isolated eigenvalue $\mu$ of $A$ of algebraic multiplicity $r$. Note that if $\lambda$ is an eigenvalue of Equation~\eqref{e:model} then $\mu:=\lambda^{-1}$ is an eigenvalue of $A$. The invariant subspace corresponding to an eigenvalue $\mu$ is defined as follows. Given a circle $\Gamma\in \mathbb{C}$ in the resolvent set of  $A$ enclosing only the eigenvalue $\mu$, set
\begin{equation}\label{e:E}
E:=\frac{1}{2\pi i}\int_\Gamma (z-A)^{-1}\,dz.
\end{equation}
We note that $E:X\rightarrow X$ is a projection operator. By $R(E)$ we denote the range of the subspace projection. The elements of $R(E)$ are generalized eigenfunctions fulfilling $a(x^j,y)=\lambda b(x^j,y)+\lambda a(x^{j-1},y)$, where $x^1$ is an eigenfunction of $A$ with eigenvalue $\lambda$, see \cite{BaOs91} page 693. It is also natural to introduce the adjoint invariant subspace projection
\begin{equation}
E^*=\frac{1}{2\pi i}\int_\Gamma (z-A^*)^{-1}\,dz.
\end{equation}
\begin{remark}
The linearization can also be done on the matrix formulation of the problem, Equation~\eqref{e:model}. The resulting matrix, corresponding to the linear map $A$, is
\begin{equation}
\left[
\begin{array}{cc}
K & 0\\
0 & M\\
\end{array}
\right]^{-1}\left[
\begin{array}{cc}
-D & -M\\
M & 0\\
\end{array}
\right]=\left[
\begin{array}{cc}
-K^{-1}D & -K^{-1}M\\
I & 0\\
\end{array}
\right].
\end{equation}
\end{remark}

\section{Generalized finite element approximation}\label{s:lod}
We present a GFEM for efficient approximation of the eigenvalues and eigenspaces of the discrete model problem present in Equation~\eqref{e:model2}. We will use the two scale decomposition introduced in \cite{MaPe11,MaPe13}.

\subsection{Orthogonal decomposition}
We want to decompose $\VhFE$ into a part of same dimension as $\VHFE$ and a remainder part. To this end we first introduce the remainder space,
$$\Vf:=\kernel \bigl(\IH\bigr)\subset \VhFE,$$
representing the fine scales in the decomposition.
The orthogonalization of the decomposition with respect to the bilinear form $(\kappa\nabla\cdot,\nabla \cdot)$ yields the definition of a modified coarse space $\VHLOD$.

Given $u\in \VhFE$, define the fine scale projection operator $\Pf u\in\Vf$ as solution to
\begin{equation}\label{e:finescaleproj}
  (\kappa\nabla \Pf u,  \nabla v)=(\kappa\nabla  u, \nabla v)\quad\text{for all }v\in \Vf.
\end{equation}
The existence of $\Pf$ follows directly from the properties of $\kappa$ and $\Vf$.

\begin{lemma}[Orthogonal two-scale decomposition]\label{l:aod}
Any function $v\in V_h$ can be decomposed uniquely into $v=v_c+v_f$, where
$$v_c=\Pc v:=(1-\Pf)v\in(1-\Pf)V_H=:\VHLOD$$
and $$v_f:=\Pf v\in\Vf=\kernel\IH.$$
The decomposition is orthogonal, $( \kappa\nabla v_c, \nabla v_f)=0$.
\end{lemma}
\begin{proof}
See Lemma 3.2 in \cite{MaPe13}.
\end{proof}

\begin{remark}
The choice of the interpolant $\IH$ affects the generalized finite element space since it affects $\Vf$. Our particular choice of $\IH$ leads to the $L^2(\Omega)$-orthogonality of $V_H$ and $\Vf$ and this can be exploited in connection with the bilinear form $b(\cdot,\cdot)$. However, other choices are possible  \cite{Pe14,BrPe14,Pe15,GaPe15}.
\end{remark}

\subsection{Construction of basis and localization}
Given the basis $\{\phi_z\}_{z\in\mathcal{N}_H}$ of $V_H$, the natural basis for $\VHLOD$ in the light of Equation~\eqref{e:finescaleproj} is given by
$$
\{\phi_z-\Pf\phi_z\}_{z\in\mathcal{N}_H};
$$
see \cite{MaPe11,MaPe13} for a detailed discussion of this construction. We note that the basis functions $\phi_z-\Pf \phi_z$ have global support in $\Omega$. However, it was proven in \cite{MaPe11} that the corrected basis functions decay exponentially  (in terms of number of coarse elements) away form the support of $\phi_z$; see Figure~\ref{fig:decay} for an illustration.
\begin{figure}[t]
\includegraphics[width=.48\textwidth]{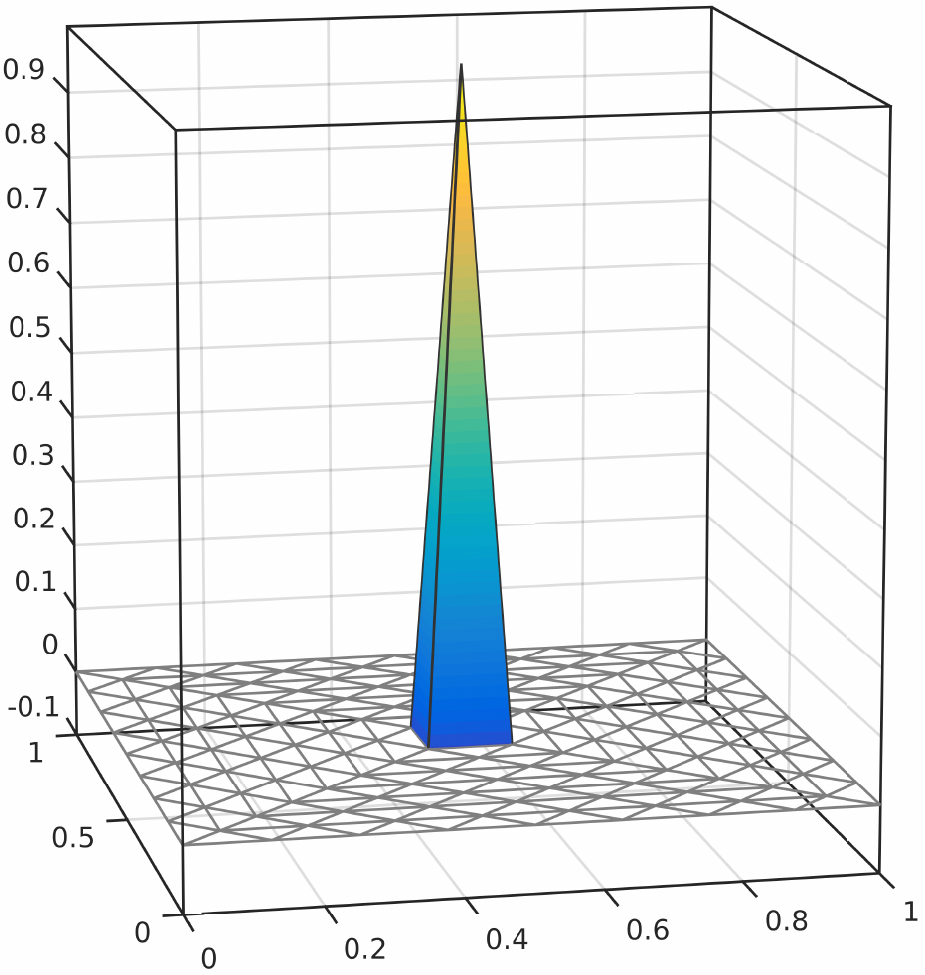}
\includegraphics[width=.48\textwidth]{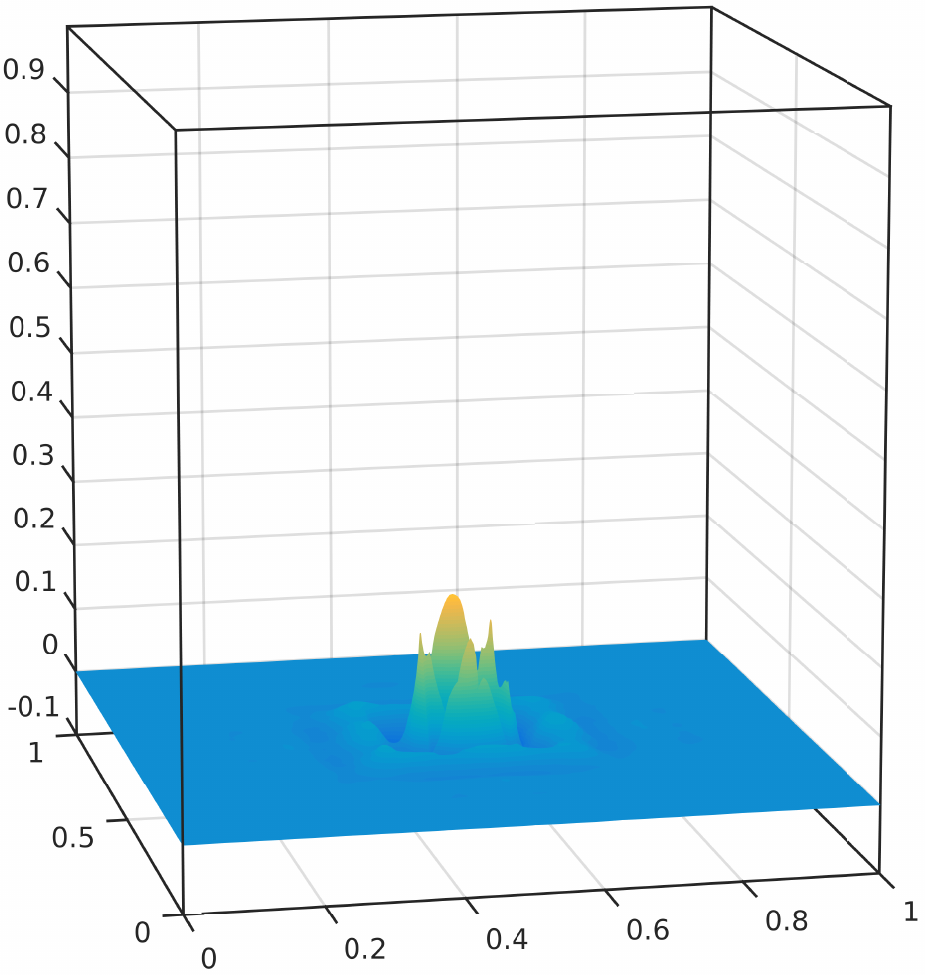}\\
\includegraphics[width=.48\textwidth]{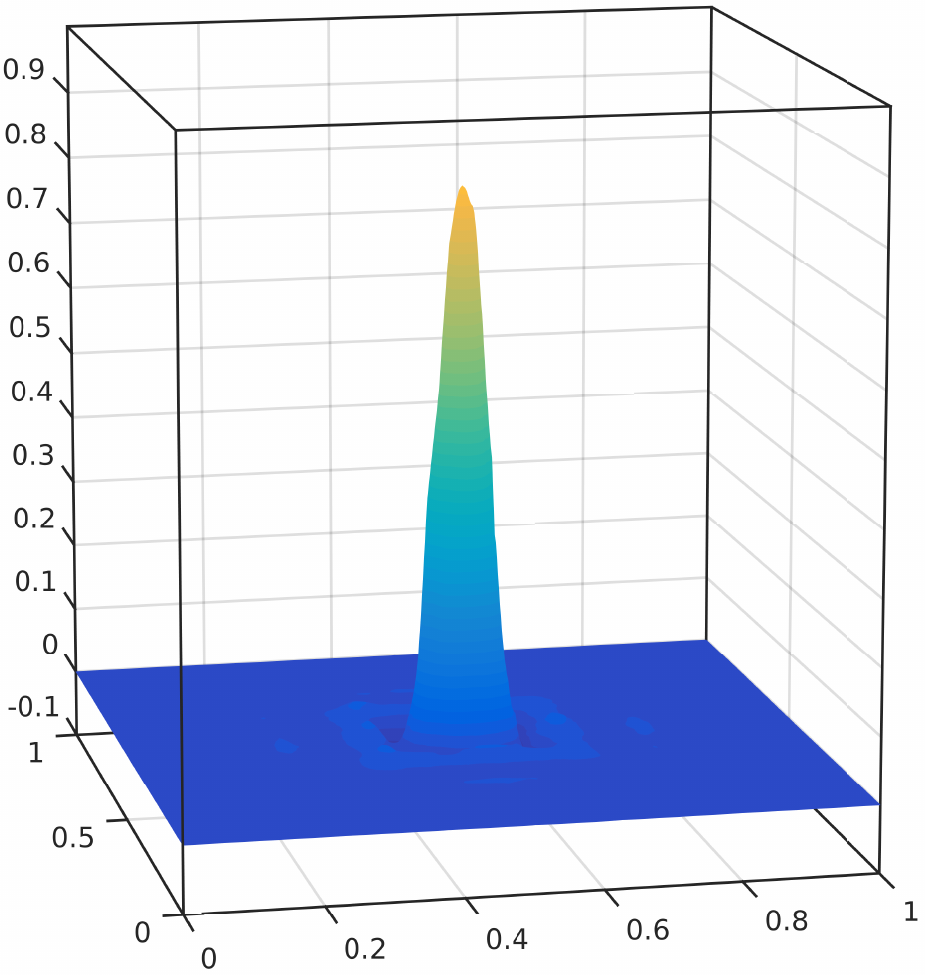}
\includegraphics[width=.48\textwidth]{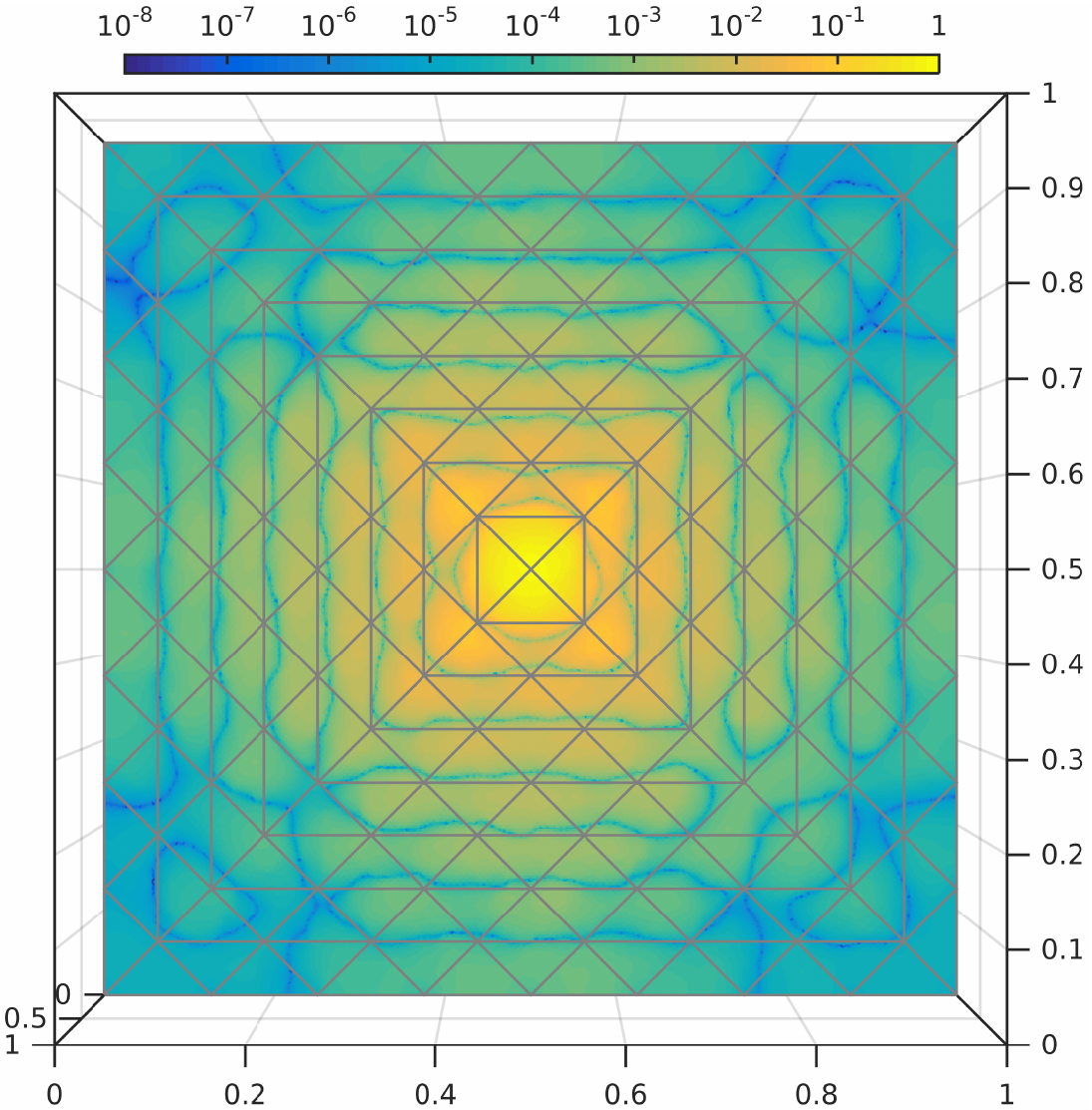}
\caption{Standard nodal basis function $\phi_z$ with respect to the coarse mesh $\tri_H$ (top left), corresponding ideal corrector $\Pf\phi_z$ (top right), and corresponding corrected basis function $\phi_z-\Pf\phi_z$ (bottom left). The bottom right figure shows a top view on the modulus of the basis function $\phi_z-\Pf\phi_z$ with logarithmic color scale to illustrate the exponential decay property. The underlying rough diffusion coefficient $A$ is depicted in Fig.~\ref{fig:numexpadata} (left).}
\label{fig:decay}       
\end{figure}
This decay allows the truncation of the computational domain of the corrector problems to local subdomains of diameter $\ell H$ roughly, where $\ell$ denotes a new discretization parameter - the localization (or oversampling) parameter.
The obvious way would be to simply replace the global domain $\Omega$ in the definition of the fine scale correction \eqref{e:finescaleproj} with suitable neighborhoods of the nodes $z$. This procedure was used in  \cite{MaPe11}. However, it turned out that it is advantageous to consider the following slightly more involved technique based on element correctors \cite{HePe13,HaMoPe15}.

We assign to any $T\in\tri_H$ its $\ell$-th order element patch
$\omega_{T,\ell}$ for a positive integer $\ell$; see Fig.~\ref{fig:patches} for an illustration.
\begin{figure}[tb]
\includegraphics[height=0.3\textwidth]{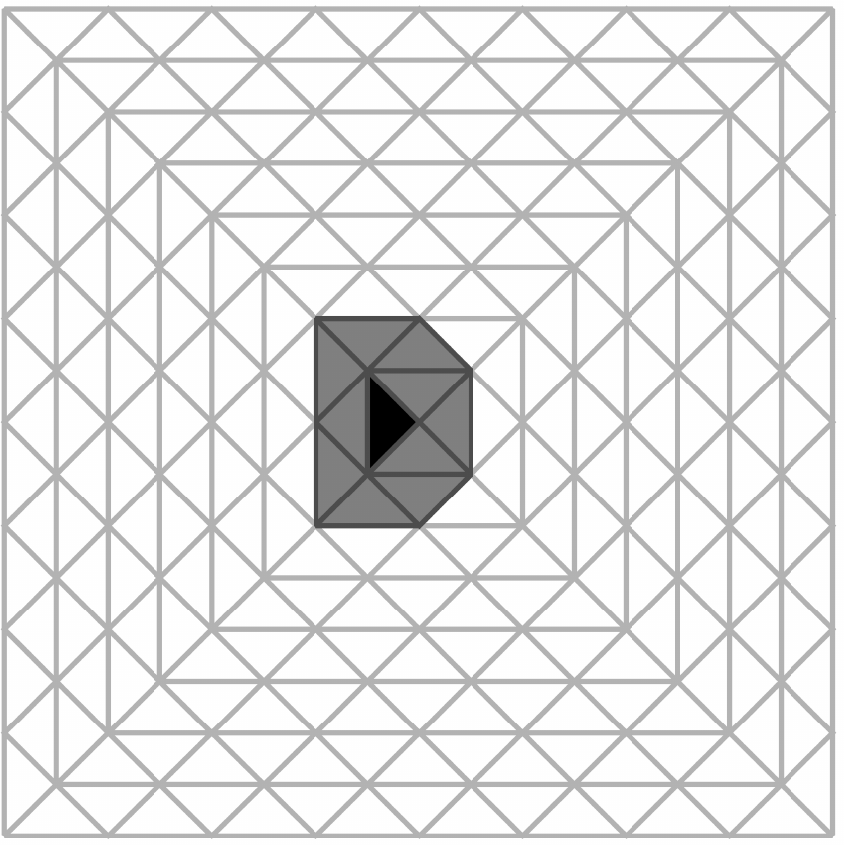}
\includegraphics[height=0.3\textwidth]{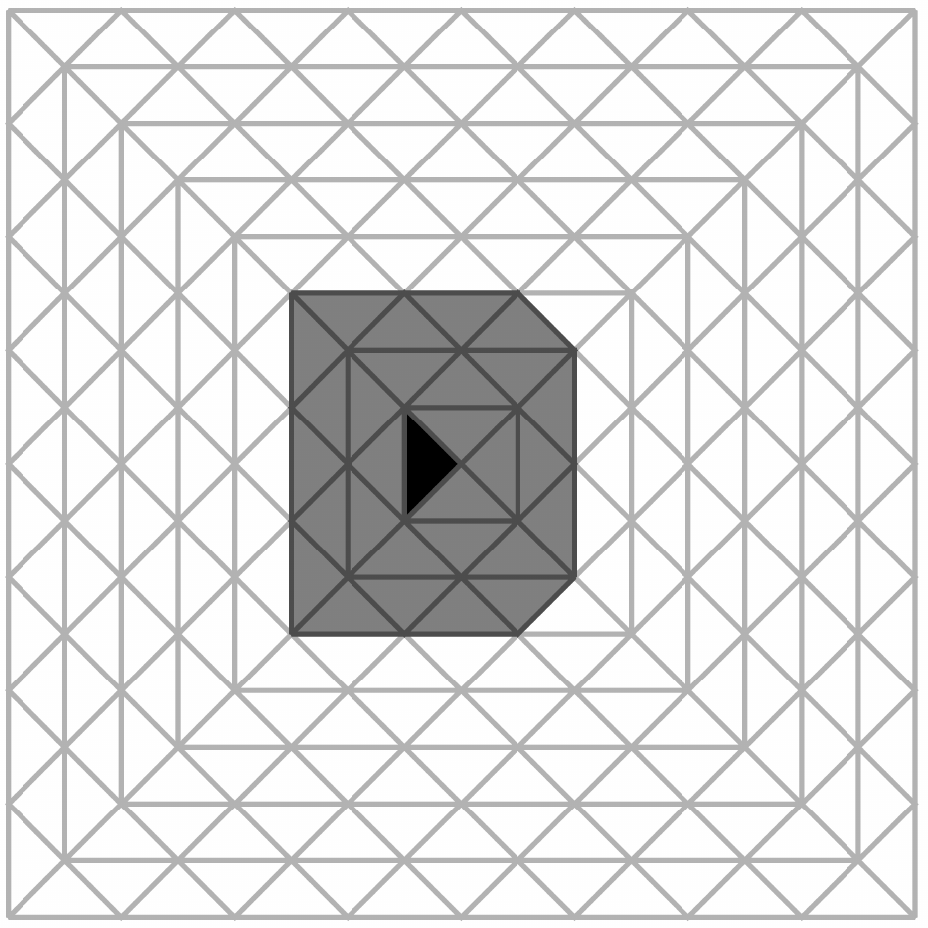}
\includegraphics[height=0.3\textwidth]{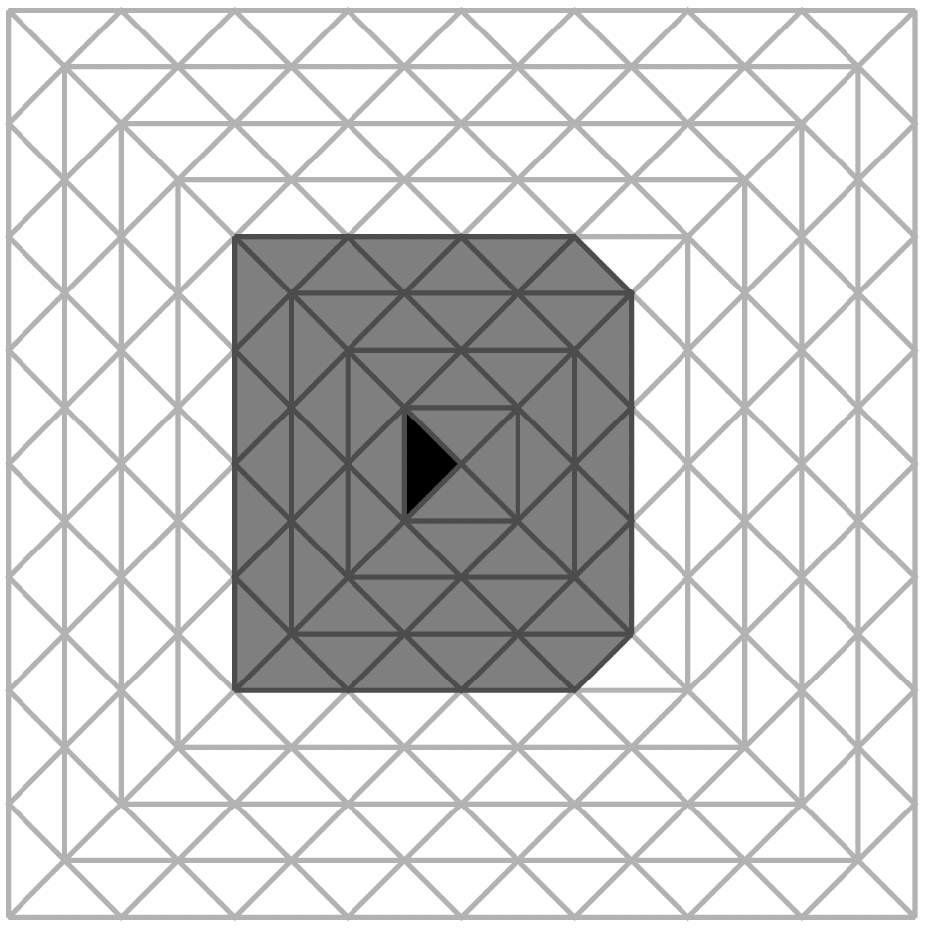}
\caption{Element patches $\omega_{T,\ell}$ for $\ell=1,2,3$ (from left to right) as they are used in the localized corrector problem \eqref{e:correctorproblemloc}.\label{fig:patches}}
\end{figure}
We introduce corresponding truncated function spaces
$$\Vf(\omega_{T,\ell})=\{v\in\Vf:\text{supp}(v)\subset\omega_{T,\ell}\}.$$

Given any nodal basis function $\phi_z\in \VHFE$,
let $\psi_{z,\ell,T}\in \Vf(\omega_{T,\ell})$ solve the localized element corrector problem
\begin{equation*}\label{e:correctorproblemloc}
(\kappa\nabla\psi_{z,\ell,T},\nabla w) = (1_T\kappa\nabla\phi_{z,\ell,T},\nabla w)
\quad\text{for all } w\in \Vf(\omega_{T,\ell}),
\end{equation*}
where $1_T$ denotes the indicator function of the element $T$.
Note that we impose homogeneous Dirichlet boundary condition on the artificial boundary of the patch which is well justified by the fast decay.
Let $\psi_{z,\ell}:=\sum_{T\in\tri_H:z\in T} \phi_{z,\ell,T}$ and
define the truncated basis function
\begin{equation*}
\phi_{z,\ell} := \phi_z - \psi_{z,\ell}.
\end{equation*}
The localized coarse space is then defined as the span of these corrected basis functions,
$$
\VHLODl:=\operatorname*{span}\{\phi_z-\psi_{z,\ell}\;\vert\;z\in\mathcal{N}_H\}.
$$
Because of the exponential decay of $\Pf\phi_z$ the $H^1(\Omega)$-error $\Pf\phi_z-\psi_{z,\ell}$ can be bounded in terms of $H^k$ (for any $k$) if $\ell=C\, k\, \log(H^{-1})$, i.e., if the diameter of the patches are of size $C\, k\, H\log(H^{-1})$ \cite{MaPe11,MaPe13,HePe13}.

\subsection{A best approximation result}
We will use the approximation properties of the space $\VHLODl$ on several occasions in the error analysis in the next section.
\begin{lemma}\label{l:best}
Let $m_1,m_2:\VhFE\times\VhFE\rightarrow \mathbb{R}$ be bounded bilinear forms and let $C,\delta_1,\delta_2\geq 0$ be such that, for all $w\in \VhFE$,
\begin{gather*}
m_1(v,w-\IH w)\leq CH^{\delta_1} \|v\|\|\sqrt{\kappa}\nabla w\|,\\
m_2(v,w-\IH w)\leq CH^{\delta_2} \|\sqrt{\kappa}\nabla v\|\|\sqrt{\kappa}\nabla w\|.
\end{gather*}
Furthermore, let $z^j\in \VhFE$ ($j=1,2$) solve
$$
(\kappa\nabla z^j,\nabla v)=m_j(f,v),\quad\text{for all }v\in\VhFE,
$$
and let $z^j_{H,\ell}\in \VHLODl$ solve
$$
(\kappa\nabla z^j_{H,\ell},\nabla v)=m_j(f,v),\quad\text{for all } v\in\VHLODl,
$$
with $\ell=C(\delta_j)\log(H^{-1})$ chosen appropriately. Then it holds that
\begin{gather*}
\|\sqrt{\kappa}\nabla(z^1-z^1_{H,\ell}\|\leq CH^{\delta_1}\|f\|,\\
\|\sqrt{\kappa}\nabla(z^2-z^2_{H,\ell}\|\leq CH^{\delta_2}\|\sqrt{\kappa}\nabla f\|.
\end{gather*}
\end{lemma}
\begin{proof}
The proof follows from the statement and proof of \cite[Theorem 4.1]{MaPe11} but, for a complete proof, see also in \cite[Lemma 3.2]{MaPr15}. The dependency of the constant $C(\delta_j)$ in the relation $\ell=C(\delta_j)\log(H^{-1})$ is also discussed there.
\end{proof}

\subsection{The proposed method}
We are ready to present the approximation of the quadratic eigenvalue problem. Again, we use an operator formulation based on $A_H,A_H^*:X\rightarrow \XHLODl=\VHLODl\times \VHLODl$. Given $x\in X$, $A_H x\in \XHLODl$ and $A^*_H x\in \XHLODl$ are characterized by
\begin{gather}
a(A_H x,y)=b(x,y)\quad\text{for all }y\in \XHLODl,\\
a(x, A^*_H y)=b(x,y)\quad\text{for all }y\in \XHLODl.
\end{gather}
We denote by $\{\mu^i_H\}_{i=1}^{r_H}$ the eigenvalues of $A_H$ that approximate an eigenvalue $\mu$ of the operator $A$ from \eqref{e:A}. The corresponding invariant subspace is denoted $R(E_H)$. Given a circle $\Gamma_H\in \mathbb{C}$ in the resolvent set of  $A_H$ enclosing the eigenvalues $\{\mu^i_H\}_{i=1}^{r_H}$, let,
\begin{equation}\label{e:EH}
E_H=\frac{1}{2\pi i}\int_{\Gamma_H} (z-A_H)^{-1}\,dz.
\end{equation}
We note that $E_H:X\rightarrow X_H$ is a projection operator, just as its reference counterpart $E$ defined in \eqref{e:E}.

\subsection{Complexity}
Let $N_H=|\mathcal{N}_H|$ be the degrees of freedom in the coarse finite element mesh $\mathcal{T}_H$ and $N_h=|\mathcal{N}_h|$ the degrees of freedom in the coarse finite element mesh $\mathcal{T}_H$. The computation of the multiscale basis $\VHLODl$, using vertex patches of diameter $H\log(H^{-1})$, is proportional to solving $N_H$ (independent) Poisson type equations of size $(N_h\log(N_H))/N_H$. Then one quadratic eigenvalue problem of size $N_H$ need to be solved. Since the convergence rate is typically high (as seen in the numerical experiments) $N_H$ can be kept small leading to a very cheap system to solve. This should be compared to solving a quadratic eigenvalue problem of size $N_h$.

\section{Error analysis}\label{s:error}
In this section we study convergence of the coarse GFE approximation to the discrete reference solution. We shall emphasize that our error analysis is fully discrete and does not rely on any regularity assumptions on the PDE eigenvalue problem in the limit $h\rightarrow 0$. In the presence of rough coefficients, this will lead to sharp rate. However, for smoother problems, the worst-case nature of our analysis is presumably a bit pessimistic.

Recall that the space $X=\VhFE\times \VhFE$ is equipped with the product norm $\|x\|^2_X=\|\sqrt{\kappa}\nabla x_1\|^2+\|x_2\|^2$.

\paragraph{{\bf Assumption B}}
We assume that the bilinear form $d$ associated with the damping is real and bounded and that there exist $C>0$ and $0<\gamma\leq 2$ such that
\begin{equation}\label{comp}
|d(v,w-\IH w)|\leq CH^\gamma \|\nabla v\|\|\nabla w\|\quad\text{for all }v,w\in \VhFE.
\end{equation}
We follow the theory presented in \cite{BaOs91}, which is in turn based on \cite{Os75}. Note that the constants in the analysis are independent of variations in $\kappa$.
\begin{lemma}\label{l:aha}
Assumptions (A) and (B) imply
\begin{gather*}
\|A-A_H\|_{X,X}\leq C H^{\min(1,\gamma)},\\
\|A^*-A^*_H\|_{X,X}\leq C H^{\min(1,\gamma)}.
\end{gather*}
Furthermore,
\begin{gather*}
\|(A-A_H)|_{R(E)}\|_{X,X}\leq CH^{\gamma}\\
\|(A^*-A^*_H)|_{R(E^*)}\|_{X,X}\leq CH^{\gamma}
\end{gather*}
\end{lemma}
\begin{proof}
 Given $x=(x_1,x_2)\in X$, $y=(y_1,y_2)=Ax\in X$ satisfies $y_2=x_1$ and
 \begin{equation}
 (\kappa\nabla y_1,\nabla v)=-d(x_1,v)-( x_2,v)\quad\text{for all }v\in \VhFE.
 \end{equation}
  We let $A_Hx=\left[\begin{array}{c}y^H_1\\ y^H_2\end{array}\right]$ where $y_2^H=Q_c x_1$, $Q_c$ being the $L^2$-projection onto $\VHLODl$, and
$y^H_1$ solves,
 \begin{equation}
 (\kappa\nabla y^H_1,\nabla v)=-d(x_1,v)-( x_2,v),\quad\forall v\in \VHLODl.
 \end{equation}
We apply Lemma \ref{l:best} twice (with $\delta_1=\gamma$ and $\delta_2=1$) and conclude
\begin{equation}
\|\nabla(y_1-y_1^H)\|\leq C  (H^{\gamma} \|\sqrt{\kappa}\nabla x_1\|+H \|x_2\|),
\end{equation}
with $C>0$ independent of $H$. Furthermore,
\begin{equation*}
\|y_2-y_2^H\|=\|x_1-Q_c x_1\|\leq C H \|\nabla x_1\|.
\end{equation*}
In summary, for any $x\in X$ it holds,
\begin{equation}
\|(A-A_H)x\|_X\leq C H^{\min(1,\gamma)} \|x\|_X.
\end{equation}
Similar arguments gives the bound for $\|A^*-A^*_H\|_{X,X}$.
This proves the first part of the lemma.

Let $\mu$ be an eigenvalue of multiplicity $r$ with corresponding eigenvector $x$ with components $x^1_1$ and $x^1_2$. Further let $x_1^j$ and $x_2^j$ denote the two components of its generalized eigenfunctions with $j\leq r$ fulfilling
\begin{equation*}
a(x^j,y)=\lambda b(x^j,y)+\lambda a(x^{j-1},y);
\end{equation*}
see \cite[p. 693]{BaOs91}. We have that
\begin{equation*}
x_2^j=\sum_{i=1}^{j-1} \lambda^{i}x_1^{j-i+1}+\lambda^{j-1}x_1^1
\end{equation*}
and also that, for all $2\leq j\leq r$,
\begin{equation*}
 \|\sqrt{\kappa}\nabla x^{j-1}_1\|\leq C \|\sqrt{\kappa}\nabla x^j_1\|.
\end{equation*}
Hence,
\begin{equation*}
\|\sqrt{\kappa}\nabla x_2^j\|\leq C \|\sqrt{\kappa}\nabla x^j_1\|.
\end{equation*}
Applying Lemma \ref{l:best} twice with $\delta_2=\gamma$ and $\delta_2=2$ and using that ${(\IH x_2,v_f)=0}$, we get for any $x\in R(E)$ with $y=Ax$ and $y_H=A_Hx$ that
\begin{multline*}
 \|\sqrt{\kappa}\nabla(y_1-y^H_1)\|^2\\
 \leq C\left(H^\gamma \|\sqrt{\kappa}\nabla x_1\|+CH^2\|\sqrt{\kappa}\nabla x_2\|\right)\|\sqrt{\kappa}\nabla(y_1-y_1^H)\|\\
 \leq C(H^\gamma+H^2) \|\sqrt{\kappa}\nabla x_1\|\|\sqrt{\kappa}\nabla (y_1-y_1^H)\|.
\end{multline*}
Since $\gamma\leq 2$, we conclude that
\begin{equation}
\|(A-A_H)x\|_X\leq C H^\gamma \|x\|_X.
\end{equation}
Similar arguments yield the bound for $\|(A^*-A^*_H)|_{R(E^*)}\|_{X,X}$.
\end{proof}
Lemma \ref{l:aha} shows that the approximate operator $A_H$ converges in norm, with rate $H^{\min(1,\gamma)}$, to $A$. Therefore also the spectrum converges. Given the circle $\Gamma$ in the complex plane only containing the isolated eigenvalue $\mu$ for sufficiently small $H$ the approximate eigenvalues $\mu_H^j$ will also be inside $\Gamma$; see \cite{We97}. However, a quantification of what sufficiently small means is difficult as it depends also on the unknown spectrum and, in particular, the sizes and the separation of the eigenvalues. For a given range of coarse discretization parameters, we will, hence, have to assume that a curve $\Gamma$ exists that contains only $\mu$ and $\mu_H^j$, $j=1,\dots,r$, or expressed in an other way, that the curve $\Gamma_H$ in Equation~\eqref{e:EH} may be chosen equal to $\Gamma$.

\paragraph{{\bf Assumption C}}
Given $h$ and an eigenvalue $\mu$, we assume that there exists $H_0>h$ such that for all $h\leq H \leq H_0$ the curve $\Gamma_H$ in Equation~\eqref{e:EH} only containing the eigenvalues $\mu_H^j$, $j=1,\dots,r$ can be chosen equal to $\Gamma$ which contains only $\mu$.

We will comment further on this assumption in the numerical examples. For the subsequent error analysis it means that the results are only valid in the regime $h\leq H \leq H_0$.

\begin{lemma}\label{l:efunc}
The Assumptions (A), (B), and (C) imply that, for all sufficiently small $H$,
 $\|(E-E_H)|_{R(E)}\|_{X}\leq C H^\gamma$, $\operatorname{dim}(R(E))=\operatorname{dim}(R(E_H))$, and
 $$\hat{\delta}(R(E),R(E_H))\leq CH^\gamma,
 $$
where the gap $\hat{\delta}(M,N)$ between $M$ and $N$ is defined as $\max(\delta(M,N),\delta(N,M))$ with $\displaystyle\delta(M,N):=\sup_{x\in M:\|x\|=1}\text{dist}(x,N)$.
\end{lemma}
\begin{proof}
From the definition of $\Gamma$ there exists a constant $C'$ such that,
\begin{equation}\label{e:resolve}
\max_{z\in \Gamma}\|(z-A)^{-1}\|_{X,X}\leq C'.
\end{equation}
Furthermore, using Lemma \ref{l:aha},
\begin{align*}
\|(z-A_H)x\|_{X}&\geq \left|\|(z-A)x\|_{X}-\|(A-A_H)x\|_X\right|\\
&\geq C'\|x\|_X-CH^{\min(1,\gamma)}\|x\|_X\\
&\geq \frac{C'}{2}\|x\|_X,
\end{align*}
which holds if $CH^{\min(1,\gamma)}<\frac{C'}{2}$.  Therefore,
we have for each $x\in R(E)$,
\begin{align}\label{e:calc}
\|(E_H-E)x\|_X&\leq (2\pi)^{-1}\int_\Gamma \|((z-A_H)^{-1}-(z-A)^{-1}x\|_X\,|dz|\\ \nonumber
&\leq (2\pi)^{-1}\int_\Gamma \|(z-A_H)^{-1}(A_H-A)(z-A)^{-1}x\|_X\,|dz|\\ \nonumber
&\leq (2\pi)^{-1}\int_\Gamma \|(z-A_H)^{-1}(A_H-A)(z-A)^{-1}x\|_X\,|dz|\\ \nonumber
&\leq C H^\gamma.
\end{align}
Here, we have used Lemma~\ref{l:aha} and $\Gamma_H=\Gamma$ (Assumption~(C)). By choosing $H$ so that $C H^\gamma\leq \frac{1}{2}$ we have
\begin{equation*}
\|E_H|_{R(E)}\|_{X,X}\leq 3/2\quad\text{and}\quad\|(E_H|_{R(E)}^{-1}\|_{X,X}\leq 2,
\end{equation*}
i.e., $E_H|_{R(E)}$ is one-one. This implies $\operatorname{dim}(R(E))=\operatorname{dim}(R(E_H))$. From Equation~\eqref{e:calc} we see that
$$
\delta(R(E),R(E_H))\leq CH^\gamma.
$$
If $CH^\gamma<1$ it follows from \cite{Ka76} (see also \cite[Theorem 6.1]{BaOs91}) that $\delta(R(E_H),R(E))=\delta(R(E),R(E_H))$ and, hence, the assertion.
\end{proof}

We are now ready to present the main theorem of the paper about the error in the eigenvalues.
\begin{theorem}\label{t:main}
Let $\mu$ be an isolated eigenvalue of $A$ of algebraic multiplicity $r$ and ascent $\alpha$, with associate invariant subspace $R(E)$. Under assumption (A), (B), and (C), for sufficiently small $H$ it holds,
\begin{equation}
|\mu-\mu^j_H|\leq C H^{\frac{2\gamma}{\alpha}},\quad j=1,\dots,r.
\end{equation}
\end{theorem}
\begin{proof}
Let $\{\phi_i\}_{i=1}^r$ span $R(E)$ and $\{\phi^*_i\}_{i=1}^r$ span $R(E^*)$, both normalized in $L^2(\Omega)$. Lemma \ref{l:efunc} is valid and therefore Theorem 7.3 in \cite{BaOs91} gives,
\begin{align*}
|\mu-\mu_H^j|^\alpha &\leq C\sum_{i,j=1}^r|a(A-A_H)\phi_i,\phi_j^*)|\\
&\qquad +
C\|(A-A_H)|_{R(E)}\|_{X,X}\|(A^*-A^*_H)|_{R(E^*)}\|_{X,X}.
\end{align*}

The second term has been bounded in Lemma \ref{l:aha}. The first term remains. Using Galerkin Orthogonality we can subtract any $y\in \XHLODl$ in the right slot. Using Lemmas \ref{l:aha} we get,
\begin{align*}
a((A-A_H)\phi_i,\phi_j^*)&\leq \|(A-A_H)|_{R(E)}\|_{X,X}\inf_{z\in \XHLODl}\sup_{x\in R(E^*)}\|x-z\|_X\\
&\leq CH^{\gamma}\inf_{z\in \XHLODl}\sup_{x\in R(E^*)}\|x-z\|_X.
\end{align*}
Here Lemma \ref{l:best} applies with $\delta_2=\gamma$. We also note that similar to the proof of the second part of Lemma \ref{l:aha} we have $\|\sqrt{\kappa}\nabla x_2\|\leq C\|\sqrt{\kappa}\nabla x_1\|$ for $x^j\in R(E^*)$ and therefore,
$$
\inf_{z\in \XHLOD}\sup_{x\in R(E^*)}\|x-z\|_X\leq CH^\gamma.
$$
The theorem follows.
\end{proof}

\begin{remark}
The theorem states a similar result as Theorem 8.3 in \cite{BaOs91} but for this specific method. Also a result similar to Theorem 8.2 in \cite{BaOs91} can be derived in the same way. It holds,
$$
\left|\mu^{-1}-\left(\frac{1}{r}\sum_{j=1}^{r}\mu^j_H\right)^{-1}\right|\leq C H^{2\gamma}.
$$
\end{remark}

\section{Numerical Experiments}\label{s:num}
In all subsequent experiments, we let $\Omega=[0,1]^2$ and we consider nested uniform rectangular triangulations (as depicted in Figure~\ref{fig:meshes}) on all scales.
\begin{figure}
\begin{center}
\includegraphics[height=.25\textwidth]{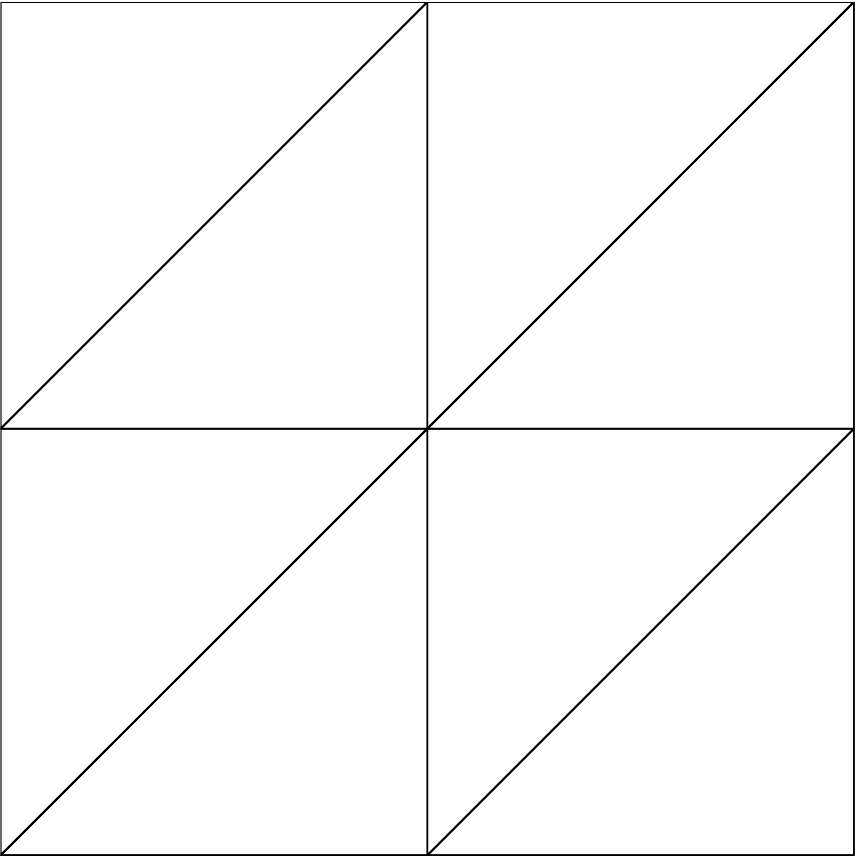}\hspace{.05\textwidth}
\includegraphics[height=.25\textwidth]{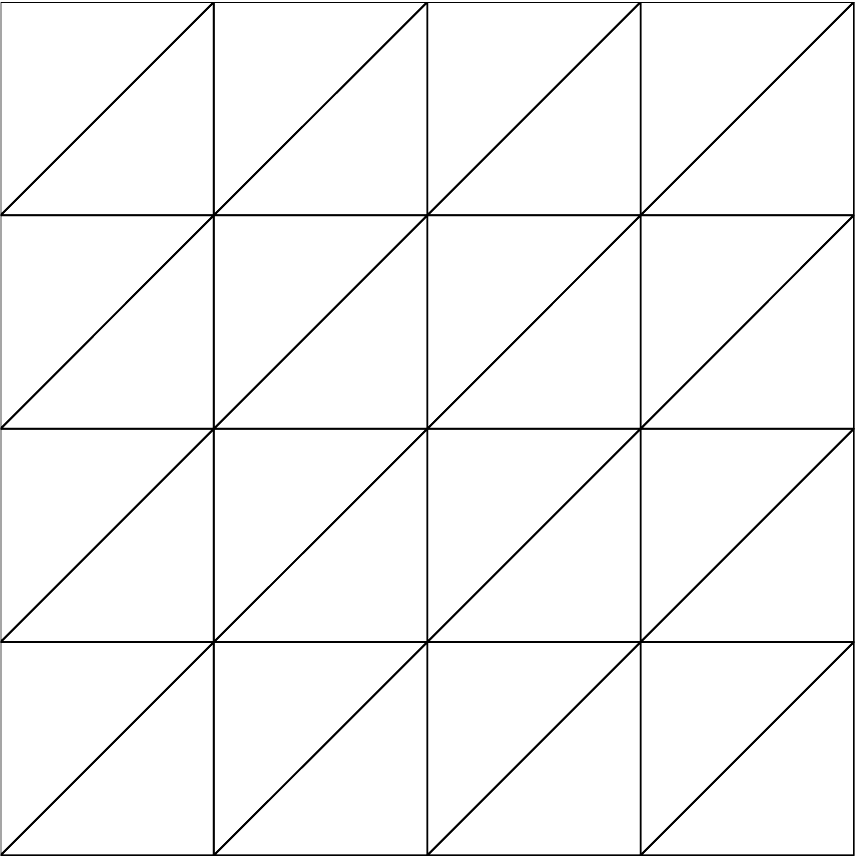}
\end{center}
\caption{Uniform triangulations of the unit square with $H=2^{-0.5},2^{-1.5}$ as they are used in the numerical experiments.}\label{fig:meshes}
\end{figure}
The fine reference scale will be fixed throughout by the choice $h=2^{-8.5}$. The coarse mesh size will typically varies, $H=2^{-1.5},2^{-2.5},2^{-3.5},
2^{-4.5},2^{-5.5}$. The corresponding finite element spaces are $V_H^{\text{FEM}}\subset V_h^{\text{FEM}}$. For each of these spaces we construct the corresponding LOD spaces $\VHLODl$ with a certain choice of the localization parameter $\ell$. Throughout the numerical experiments, we plot the relative error for the eight smallest (magnitude) eigenvalues. For the solution of the algebraic eigenvalue problems, we use the MATLAB built-in eigenvalue solver \texttt{eigs} with tolerance \texttt{1e-10}.

\subsection{Smooth mass-type damping}\label{ss:numexpa}
We consider damping by a modified mass matrix associated with the bilinear form
\begin{equation}\label{e:damp}
d(v,w)=\int_\Omega (1+sin(10x_1))v(x_1,x_2)\,w(x_1,x_2)\,d(x_1,x_2).
\end{equation}
Clearly,
$$|d(v,w-\IH w)|\leq CH^2\|\sqrt{\kappa}\nabla v\|\|\sqrt{\kappa}\nabla w\|,$$ i.e., Assumption (B) holds with $\gamma=2$.
\begin{figure}
 \begin{center}
   \includegraphics[height=.25\textwidth]{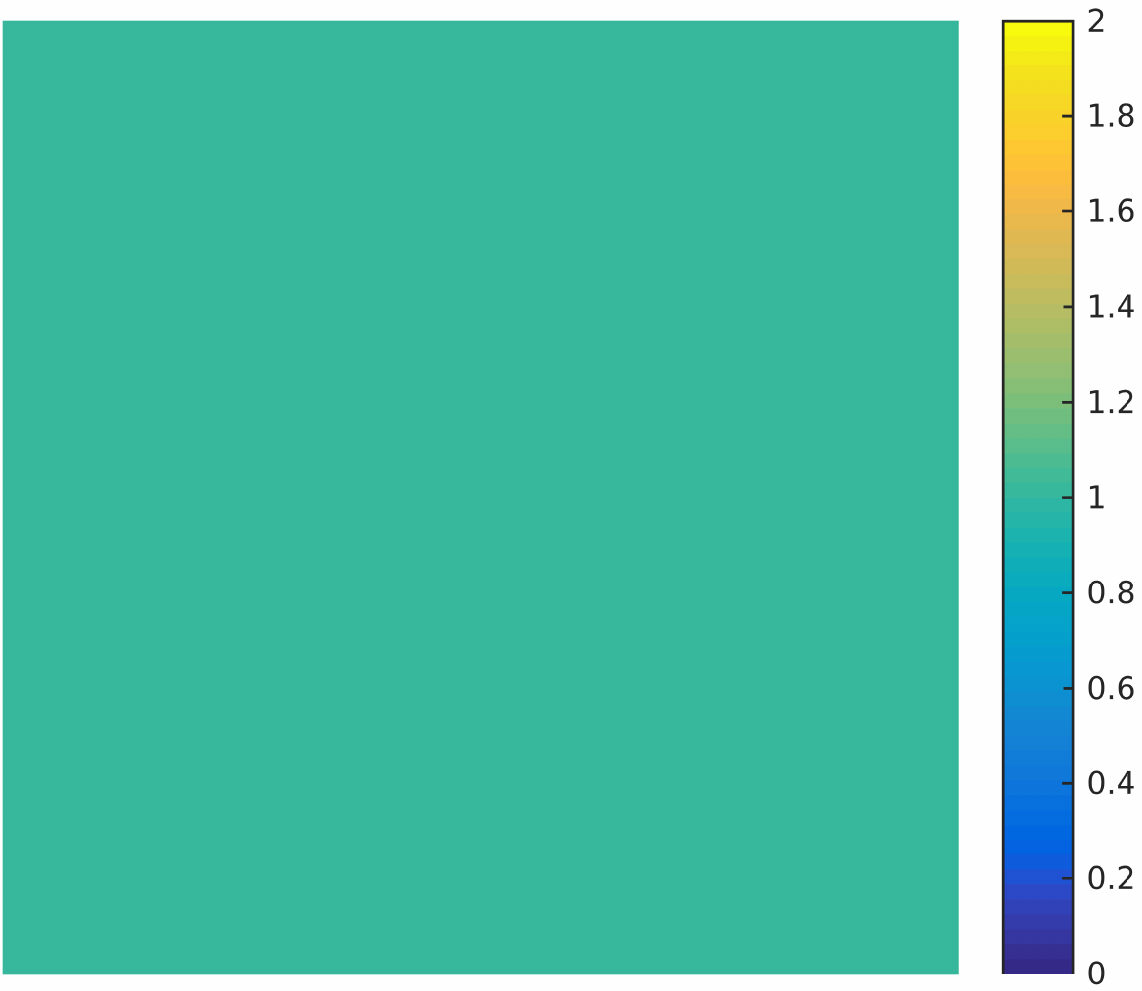}
   \includegraphics[height=.25\textwidth]{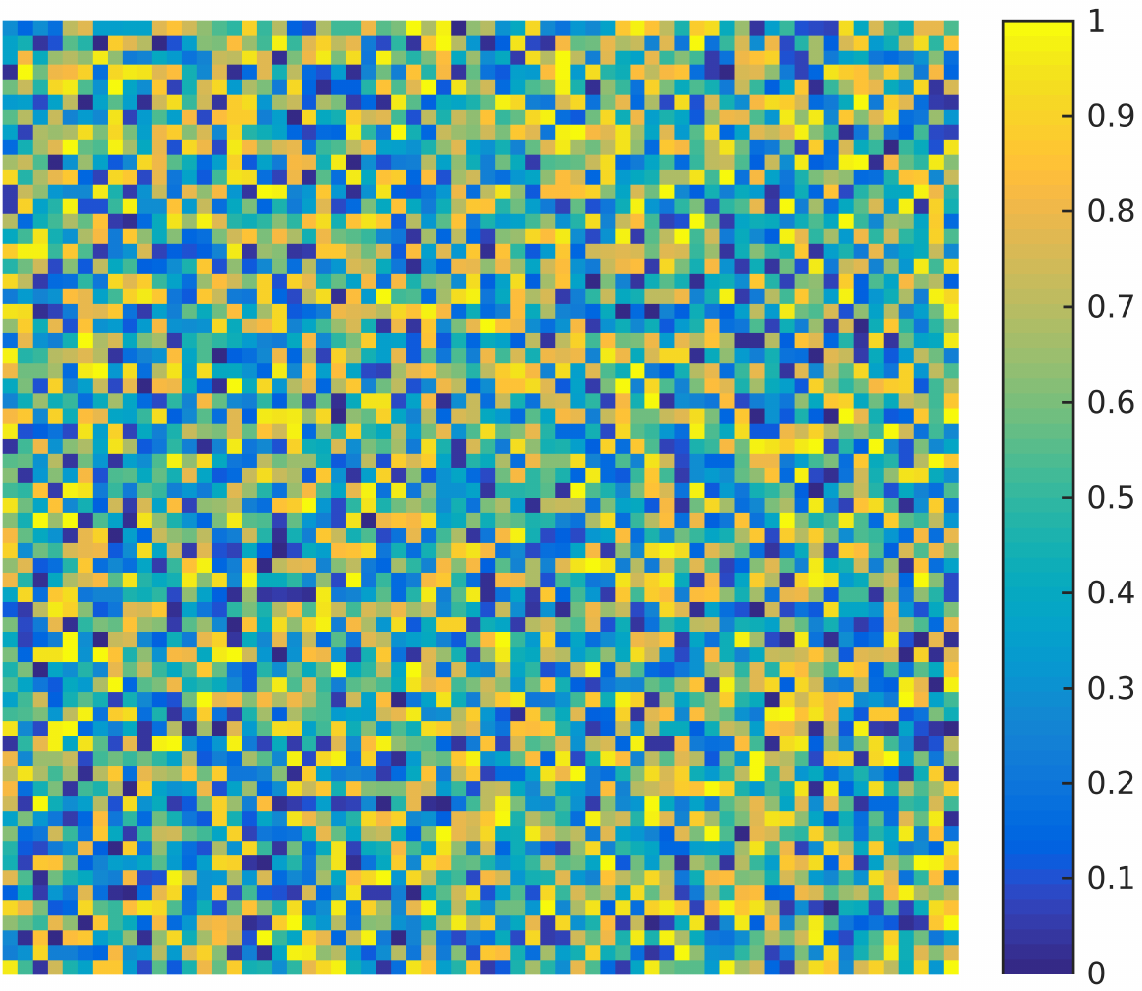}
   \includegraphics[height=.25\textwidth]{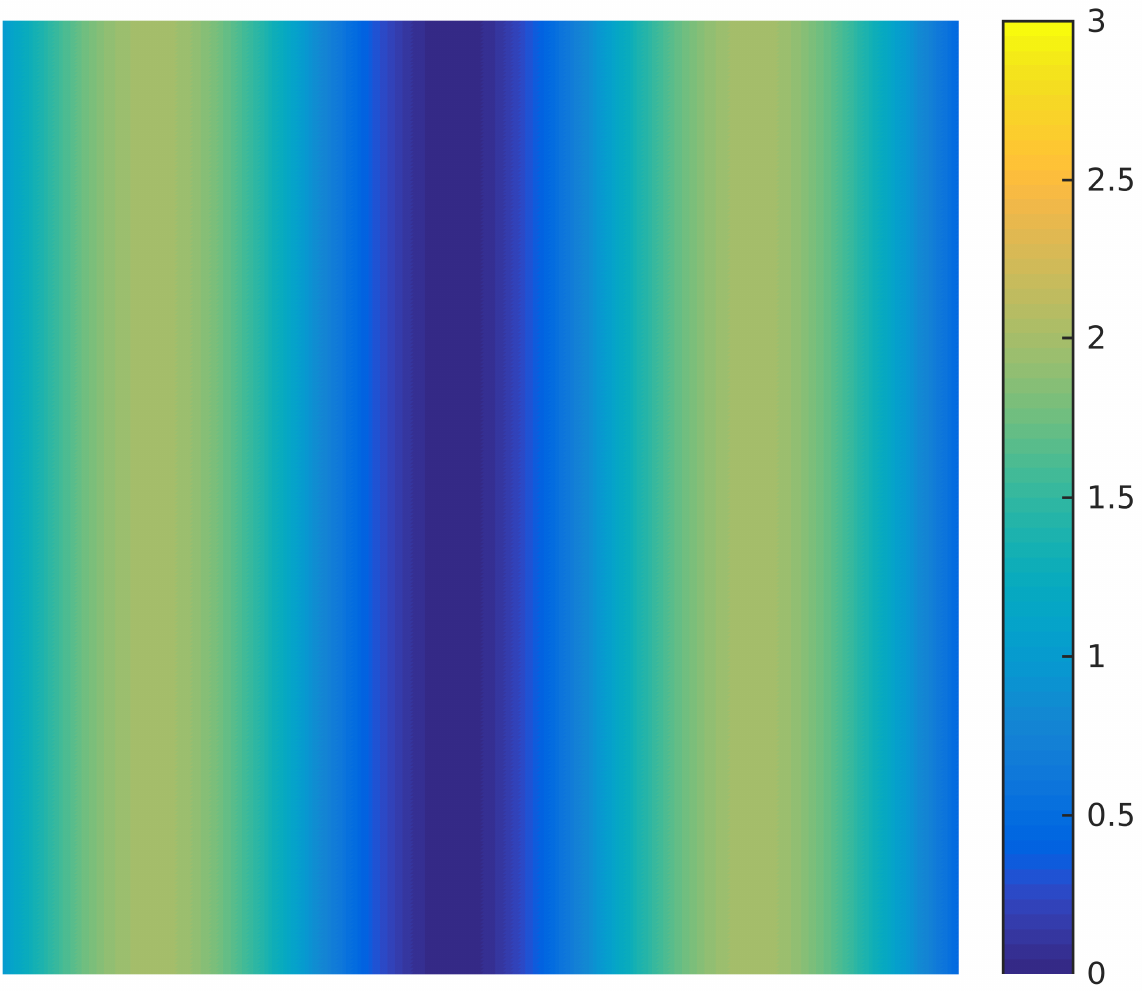}
\caption{Coefficients in the numerical experiments of Section~\ref{ss:numexpa}. Left: Constant diffusion coefficient. Middle: Rapidly varying diffusion diffusion. Right: Weight function for mass-type damping.}\label{fig:numexpadata}
  \end{center}
\end{figure}

\begin{figure}
 \begin{center}
   \includegraphics[width=.49\textwidth]{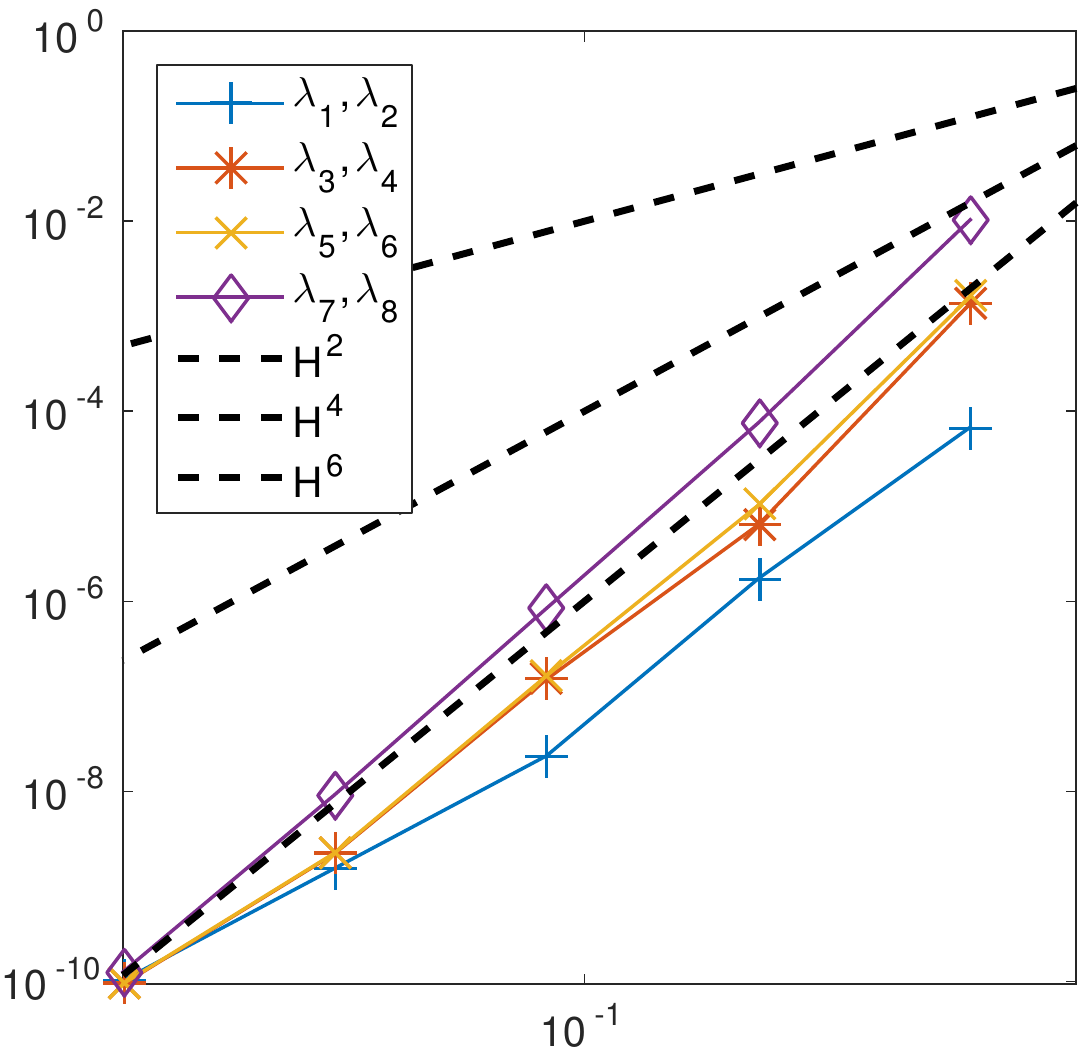}\vspace{.02\textwidth}
   \includegraphics[width=.49\textwidth]{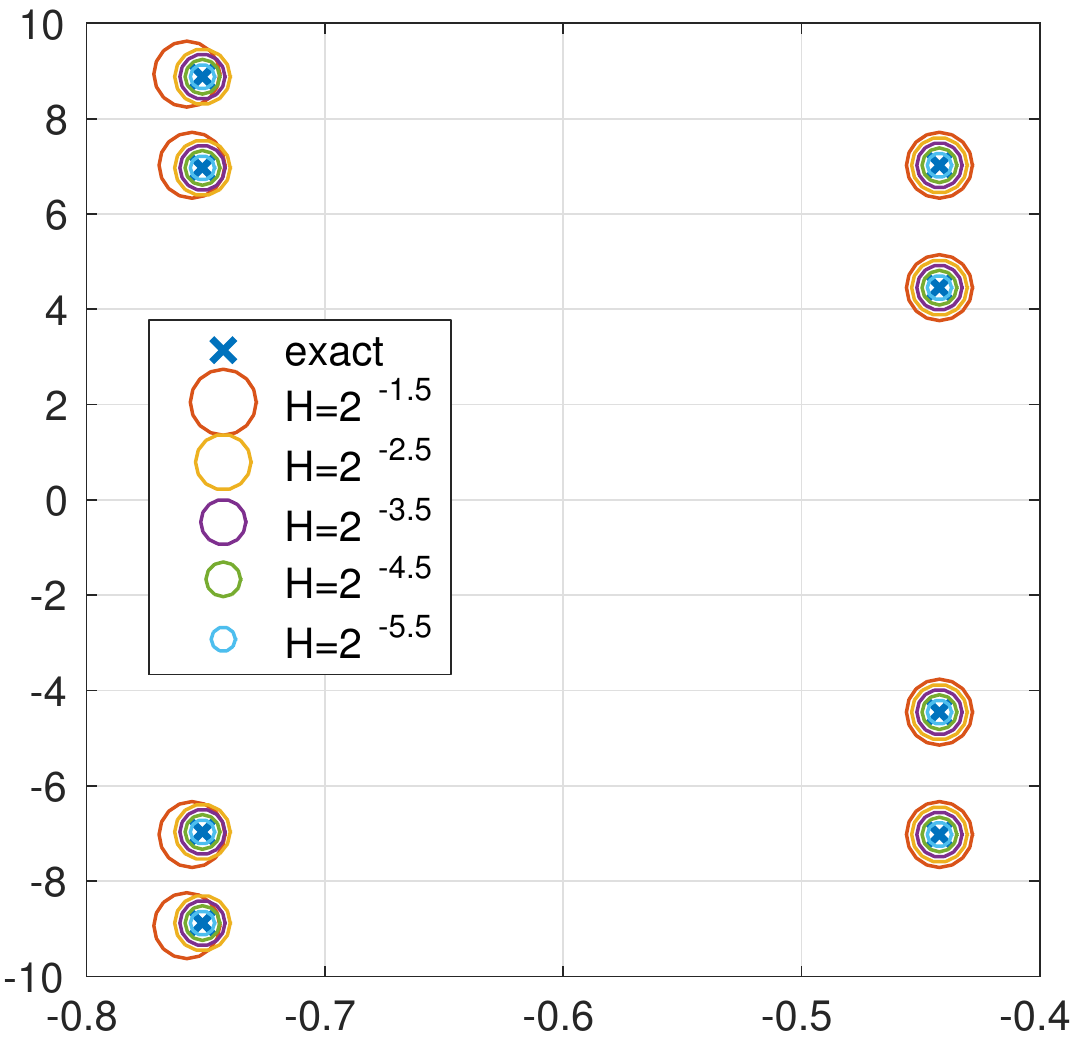}
\caption{First numerical experiment in Section~\ref{ss:numexpa} for constant diffusion and smooth damping weight depicted in Fig.~\ref{fig:numexpadata} (left and right). Left: Relative eigenvalue errors vs. coarse mesh size $H$ with localization parameter $\ell = \lceil 3\log(H^{-1})\rceil$. Right: Illustration of eigenvalue convergence as $H$ decreases.}\label{fig:numexp1}
  \end{center}
\end{figure}

To begin with, let $\kappa=1$. Figure~\ref{fig:numexp1} shows the convergence rate of the 8 smallest (complex conjugate) eigenvalues as a function of the coarse mesh size. We observe a convergence rate of $H^6$ which is more than we expect from the theory ($H^{2\gamma}=H^4$ for simple eigenvalues). This may be related to the high regularity of the underlying PDE eigenvalue problem which we do not take advantage of in the analysis.

In the second numerical example we let $\kappa$ be a piecewise constant function on a $64\times 64$ uniform grid depicted in Figure~\ref{fig:numexpadata}. In every square element we pick on value from the uniform distribution $U([0.003,1])$. This gives a deterministic rapidly varying diffusion coefficient with aspect ratio over $300$, see Figure \ref{fig:numexpadata} (middle). We keep the rest of the setup the same, i.e., $d$ is still chosen as in Equation~\eqref{e:damp}. In Figure \ref{fig:numexp2}, we see that we get the $H^4$ convergence as predicted by the theory. The same convergence rate was also detected for the corresponding linear eigenvalue problem in \cite{MaPe13}.
\begin{figure}
 \begin{center}
   \includegraphics[width=.49\textwidth]{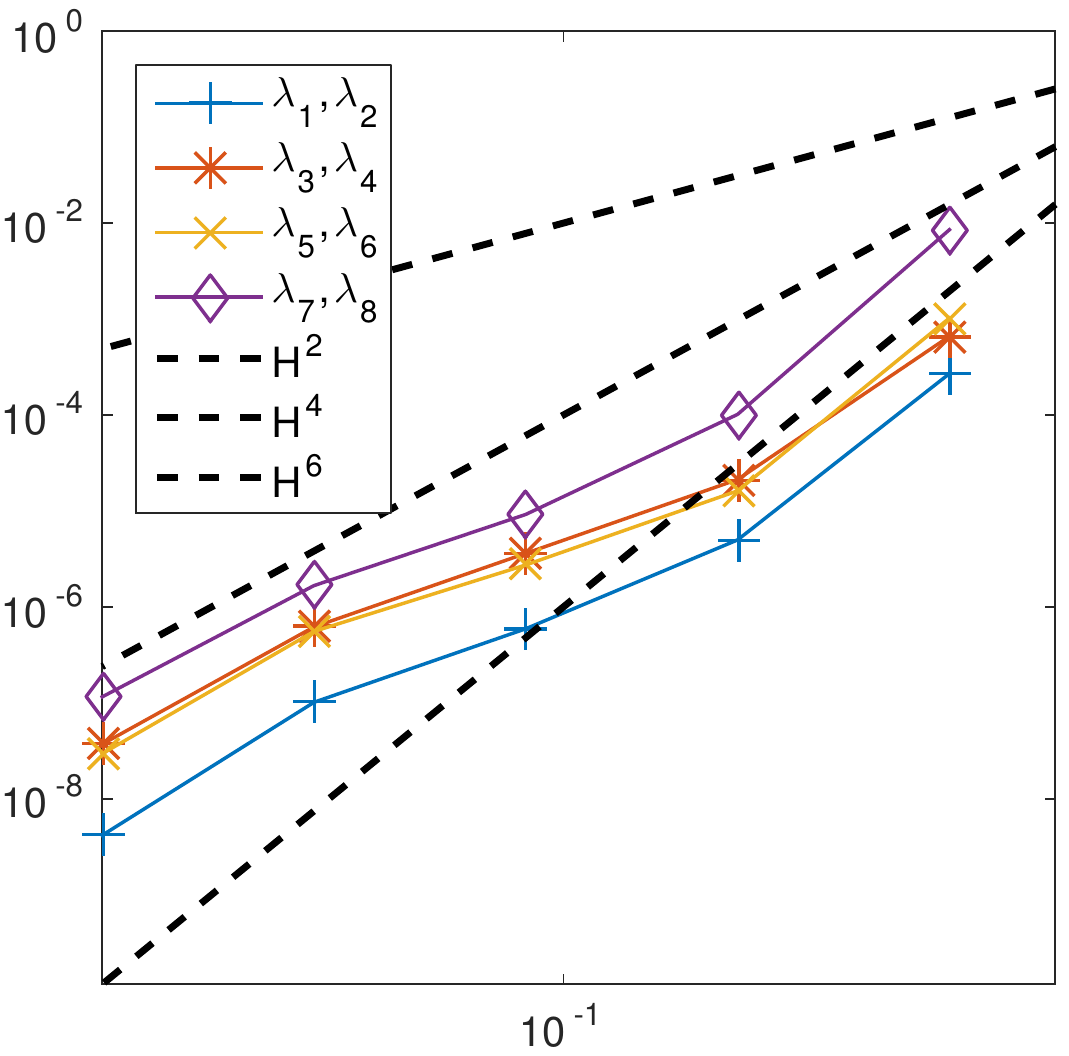}\vspace{.02\textwidth}
   \includegraphics[width=.49\textwidth]{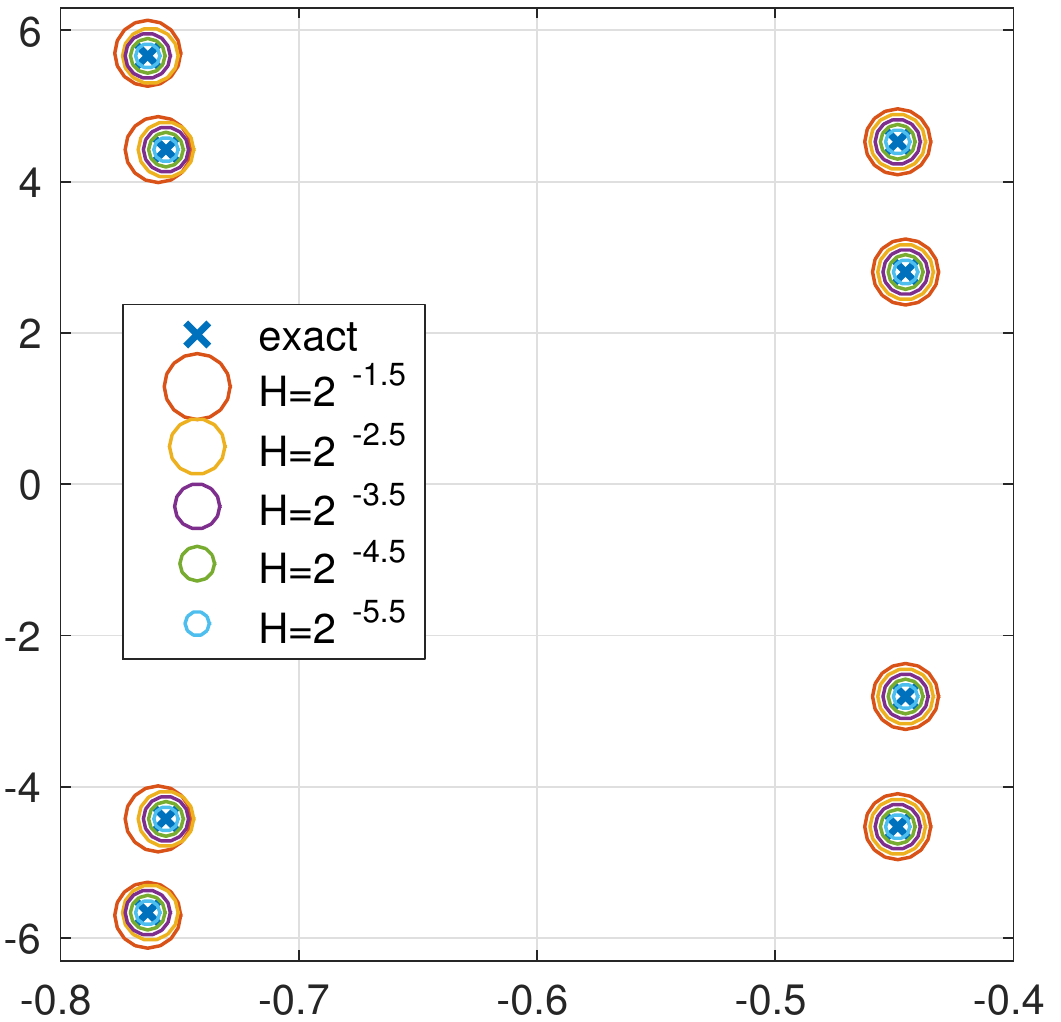}
\caption{Second numerical experiment in Section~\ref{ss:numexpa} for rough diffusion and smooth damping weight depicted in Fig.~\ref{fig:numexpadata} (middle and right). Left: Relative eigenvalue errors vs. coarse mesh size $H$ with localization parameter $\ell = \lceil 3\log(H^{-1})\rceil$. Right: Illustration of eigenvalue convergence as $H$ decreases.}\label{fig:numexp2}
  \end{center}
\end{figure}

\subsection{Discontinuous mass-type damping with composite material configuration}\label{ss:numexpb}
\begin{figure}
 \begin{center}
   \includegraphics[height=.25\textwidth]{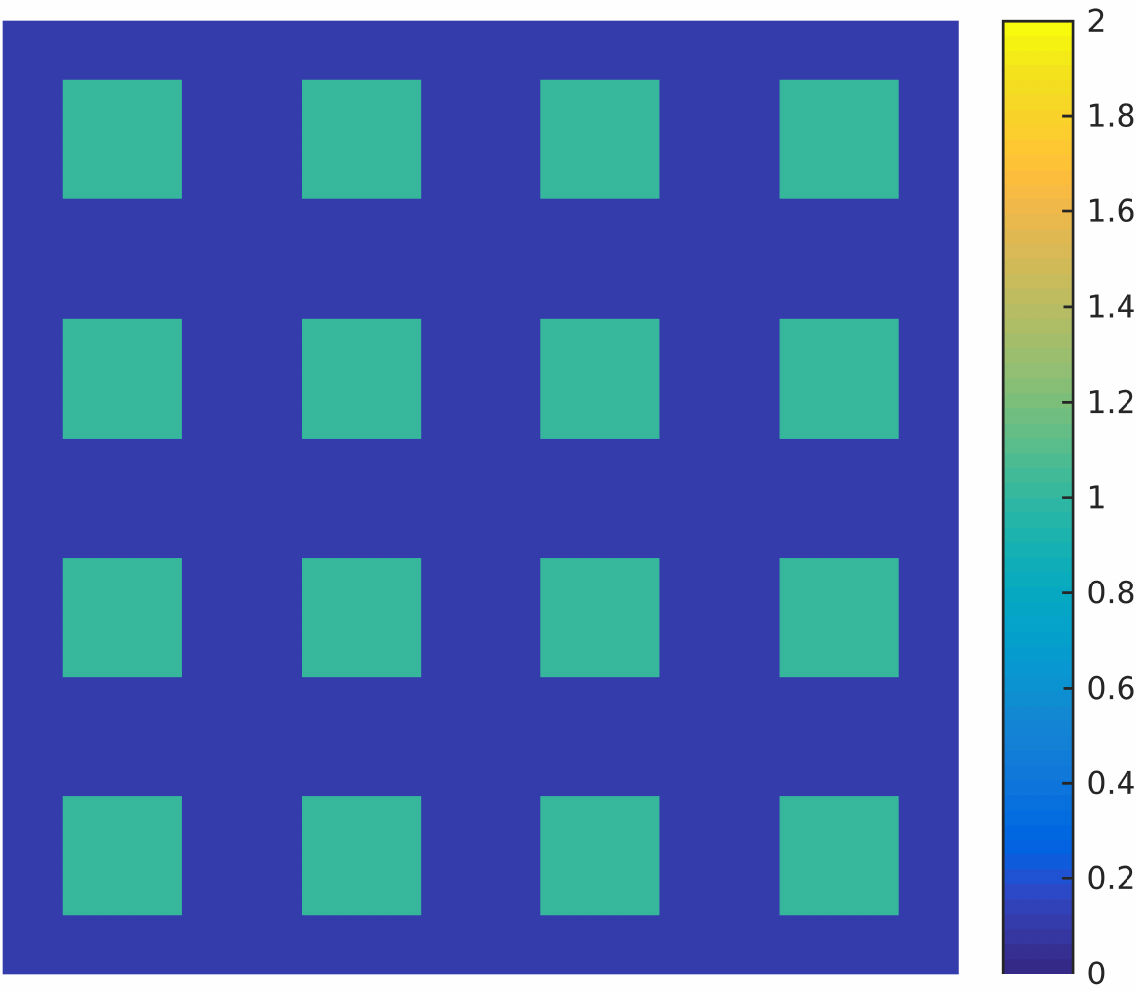}
   \includegraphics[height=.25\textwidth]{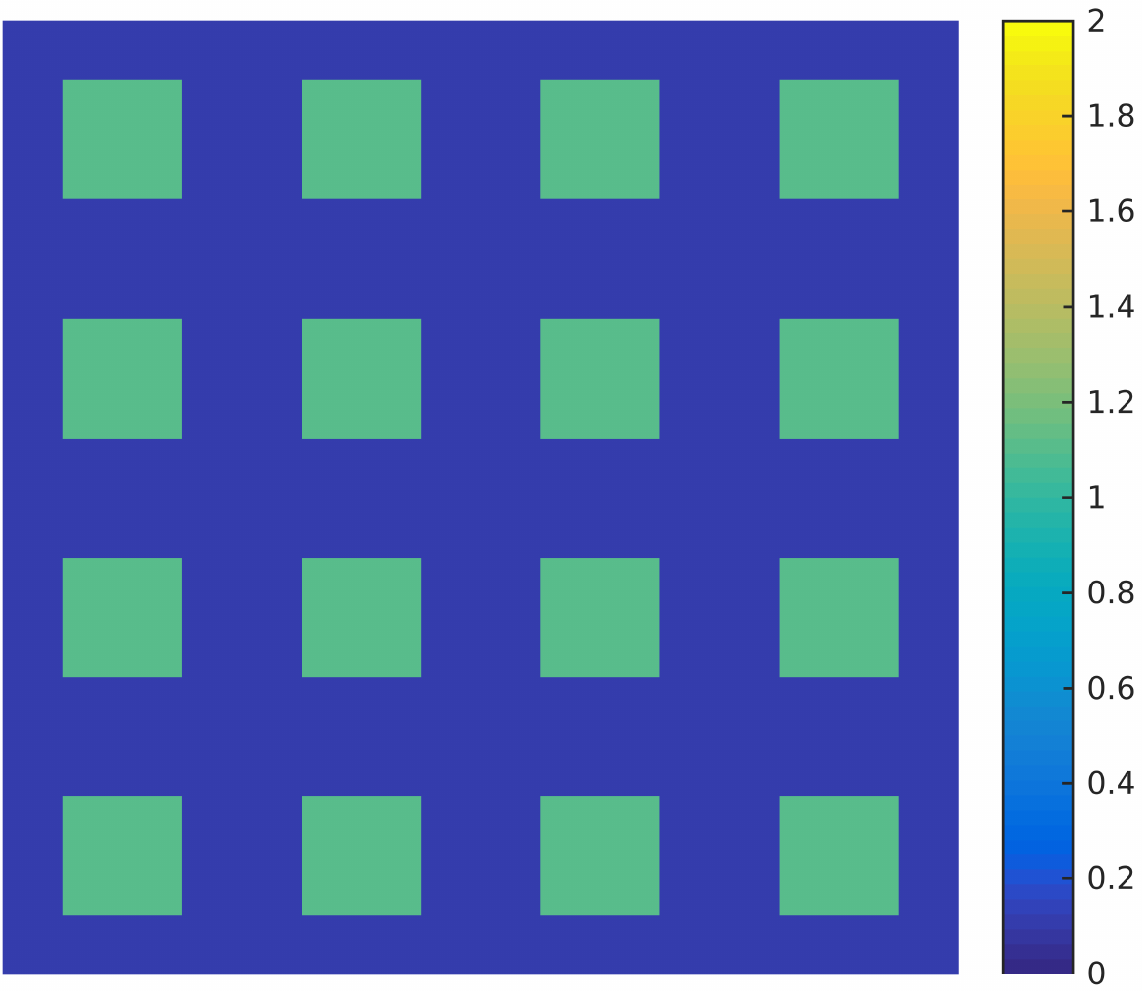}
   \includegraphics[height=.25\textwidth]{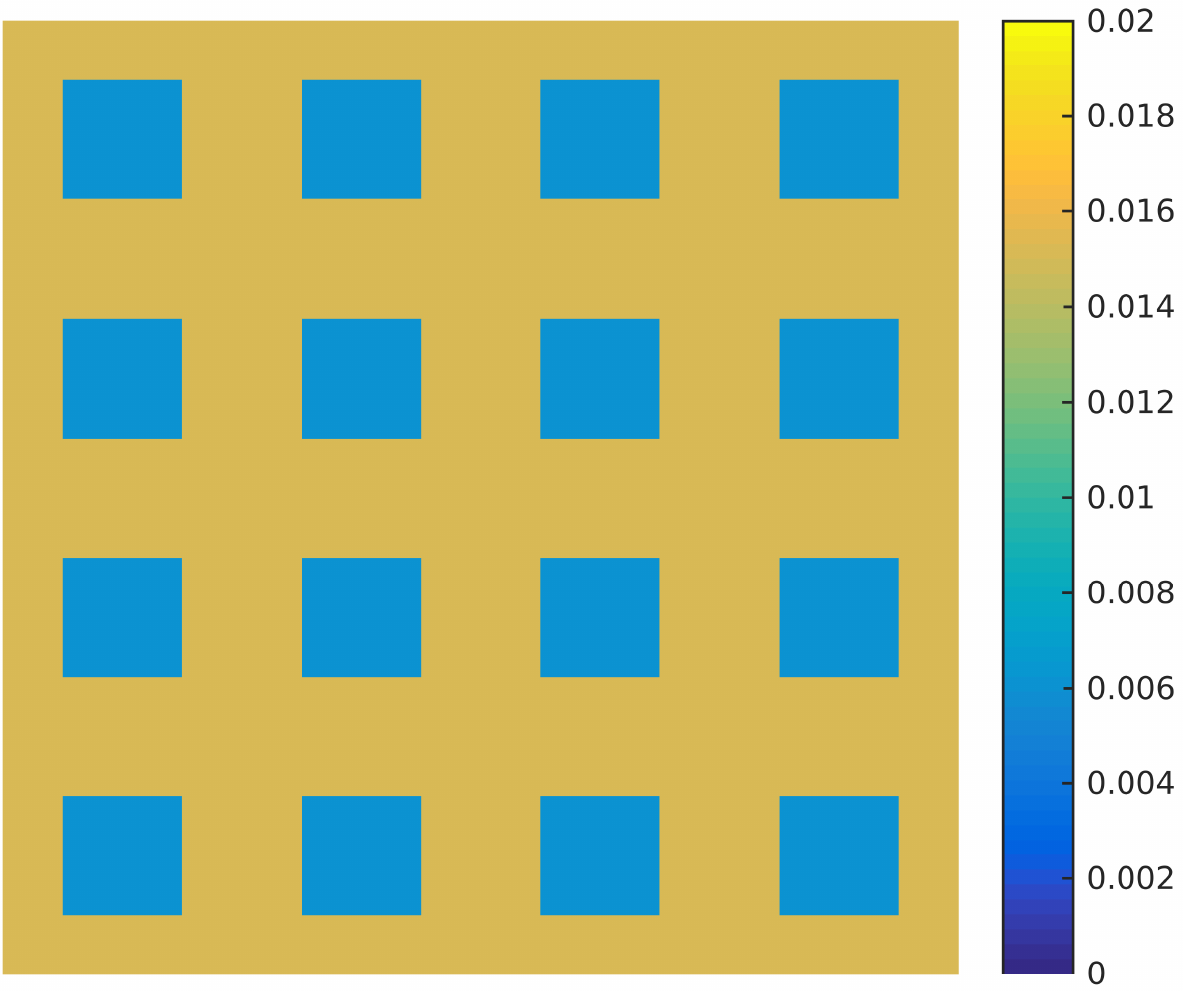}
\caption{Coefficients in the numerical experiments of Sections~\ref{ss:numexpb}--\ref{ss:numexpc}. Left: Discontinuous diffusion coefficient. Middle: Discontinuous weight function for mass-type damping. Right: Discontinuous weight function for stiffness-type damping.}\label{fig:numexpbdata}
  \end{center}
\end{figure}
\begin{figure}
 \begin{center}
   \includegraphics[width=.502\textwidth]{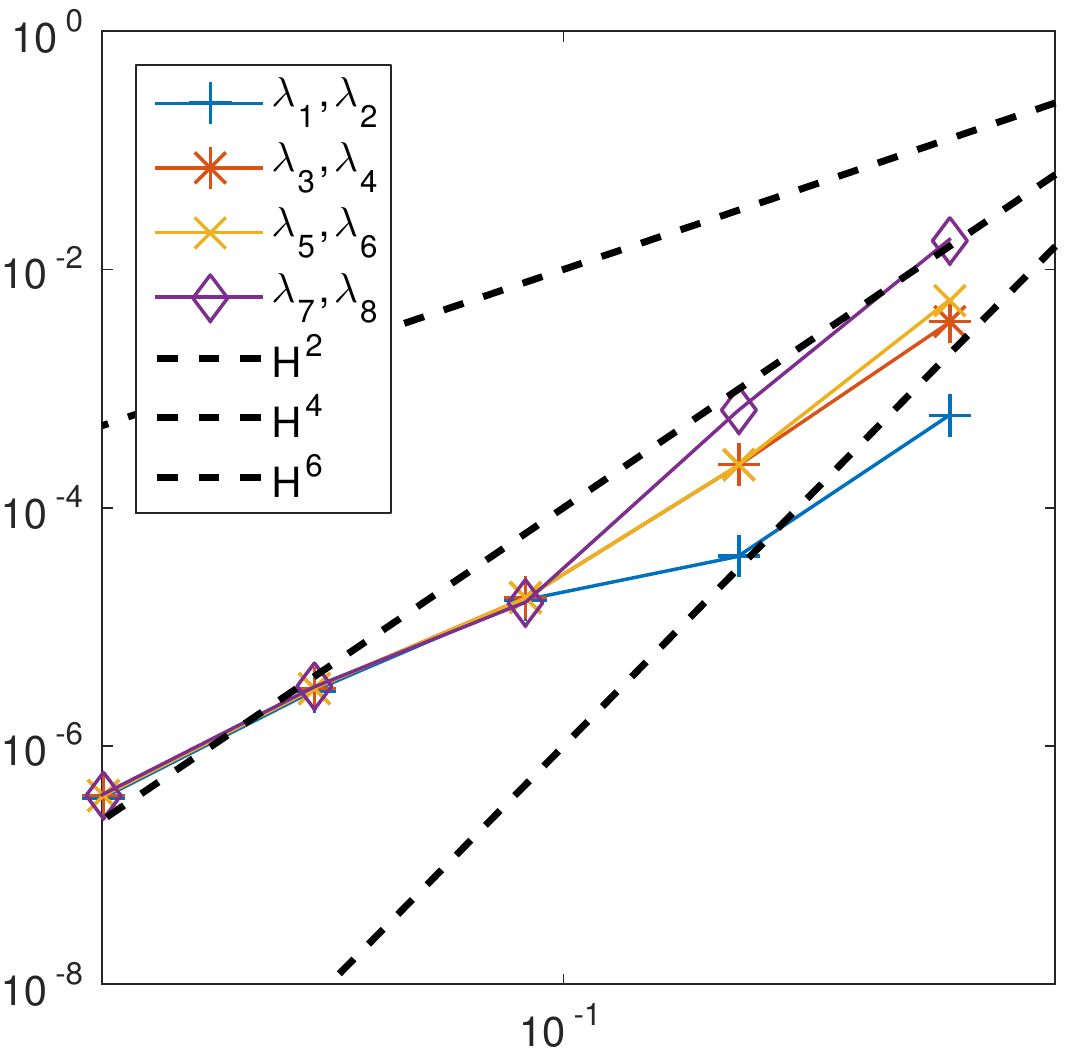}\vspace{.02\textwidth}
   \includegraphics[width=.48\textwidth]{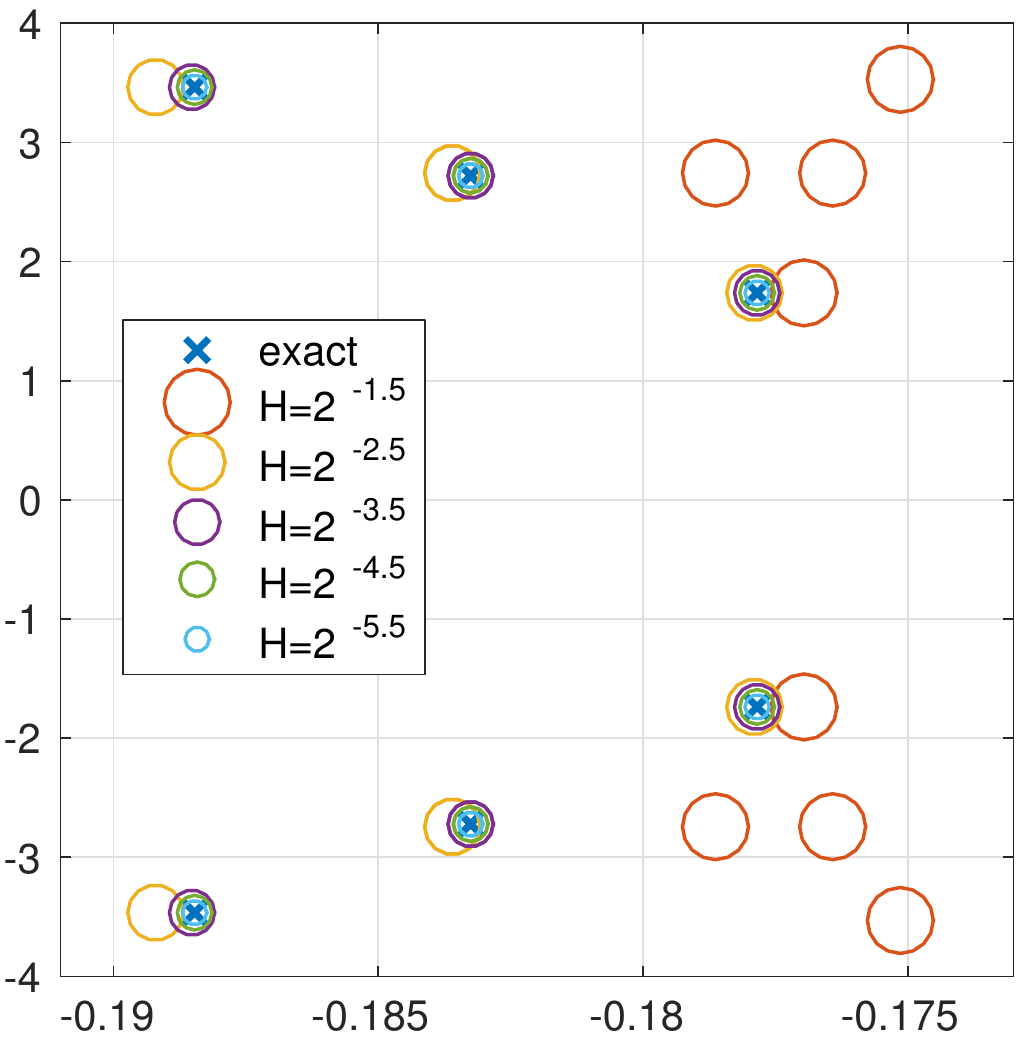}
\caption{Numerical experiment in Section~\ref{ss:numexpb} for discontinuous diffusion and discontinuous mass-type damping depicted in Fig.~\ref{fig:numexpbdata} (left and middle). Left: Relative eigenvalue errors vs. coarse mesh size $H$ with localization parameter $\ell = \lceil 3\log(H^{-1})\rceil$. Right: Illustration of eigenvalue convergence as $H$ decreases. Note that the eigenvalues $\lambda_3,\lambda_5$ (resp. $\lambda_4,\lambda_6$) are clustered and can hardly be distinguished in the plot.}\label{fig:numexp3}
  \end{center}
\end{figure}
We now consider a situation where $\kappa$ jumps between two values in the domain, namely,
$$
\kappa=\left\{\begin{array}{cl}1& \text{in }\Omega_1,\\
\alpha& \text{in }\Omega\setminus\Omega_1.\end{array}\right.
$$
This models a composite with two different materials. We pick the domain $\Omega$ to be the union of a periodic array of square inclusion as depicted in Figure~\ref{fig:numexp3} (left). We assume that the damping has the same structure represented by the coefficient
$$
\tilde{\kappa}=\left\{\begin{array}{cl}\beta& \text{in }\Omega_1,\\
\mu& \text{in }\Omega\setminus\Omega_1,\end{array}\right.
$$
and we  study mass type damping represented by the $\tilde{\kappa}$-weighted $L^2$ scalar product
$$
d(v,w)=(\tilde{\kappa}v,w).
$$
Note that in this case,
$$d(v,w-\IH w)\leq CH \|v\|\|\sqrt{\kappa}\nabla w\|,$$
i.e., $\gamma=1$. Hence, our theory predicts only convergence of order $H^2$. We can not immediately get a higher rate since $\tilde{\kappa}$ is not differentiable.
We let $\alpha=\mu=0.1$ and $\beta=1.1$. This is an academic choice used to test the numerical methods presented in the paper.
Figure \ref{fig:numexp3} shows the convergence of the error as the dimension of the coarse multiscale space depending on the coarse mesh size. The rate increases with the size of the system and we still observe $H^4$ convergence. 

In the error analysis, we have assumed that the eigenvalues are isolated and that the approximate eigenvalues are found in a neighborhood of the exact ones. This was  formulated in Assumption (C) above. While this assumption could be verified a posteriori in the previous experiments with semi-simple and well-separated eigenvalues (cf. Figures~\ref{fig:numexp1}(right), \ref{fig:numexp2}(right)), we now observe clustered eigenvalues and the violation of Assumption (C) on the coarsest meshes (cf. Figure~\eqref{fig:numexp3}(right)). Still, the quality of the approximate spectrum is comparable to the previous experiments.

\subsection{Discontinuous stiffness-type damping}\label{ss:numexpc}
\begin{figure}
 \begin{center}
   \includegraphics[width=.502\textwidth]{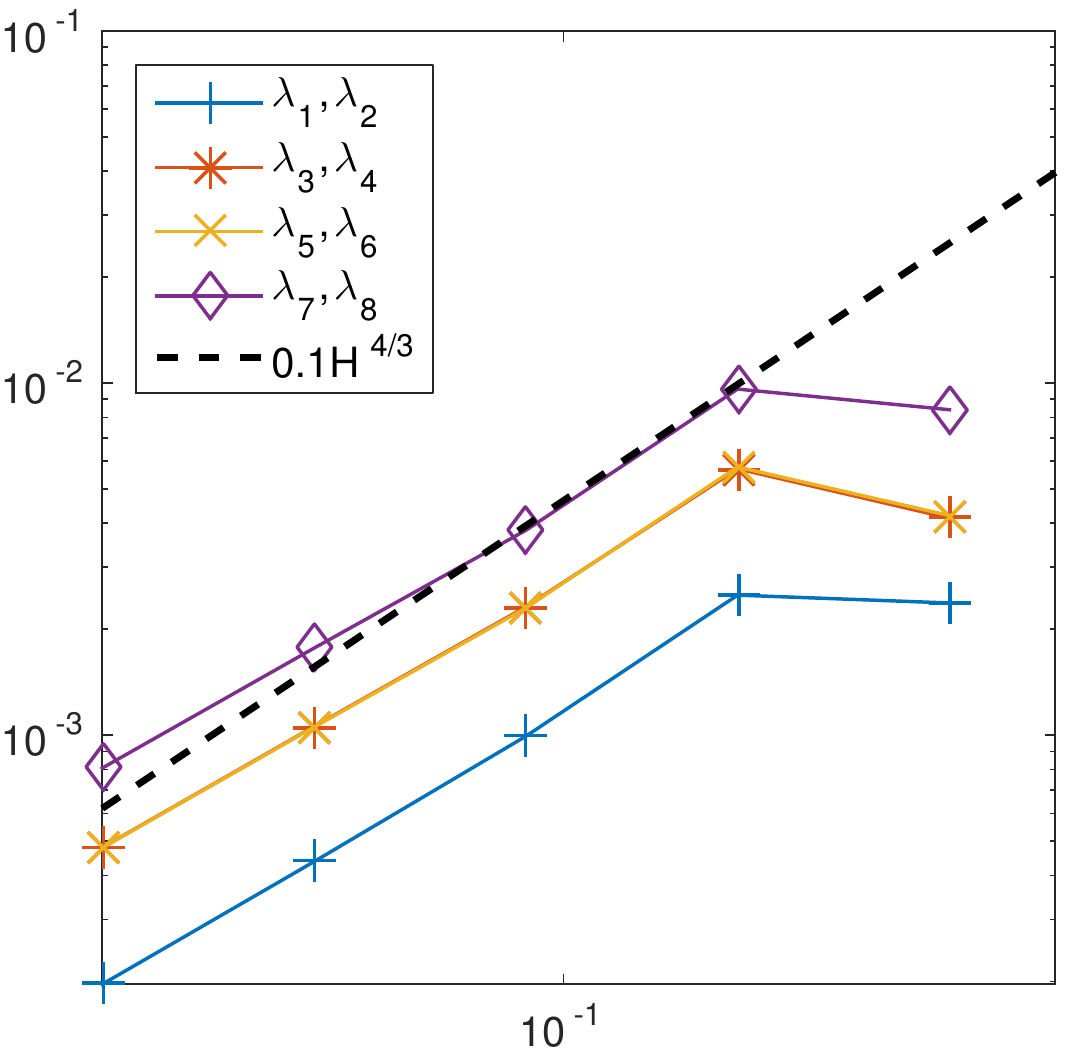}\vspace{.02\textwidth}
   \includegraphics[width=.48\textwidth]{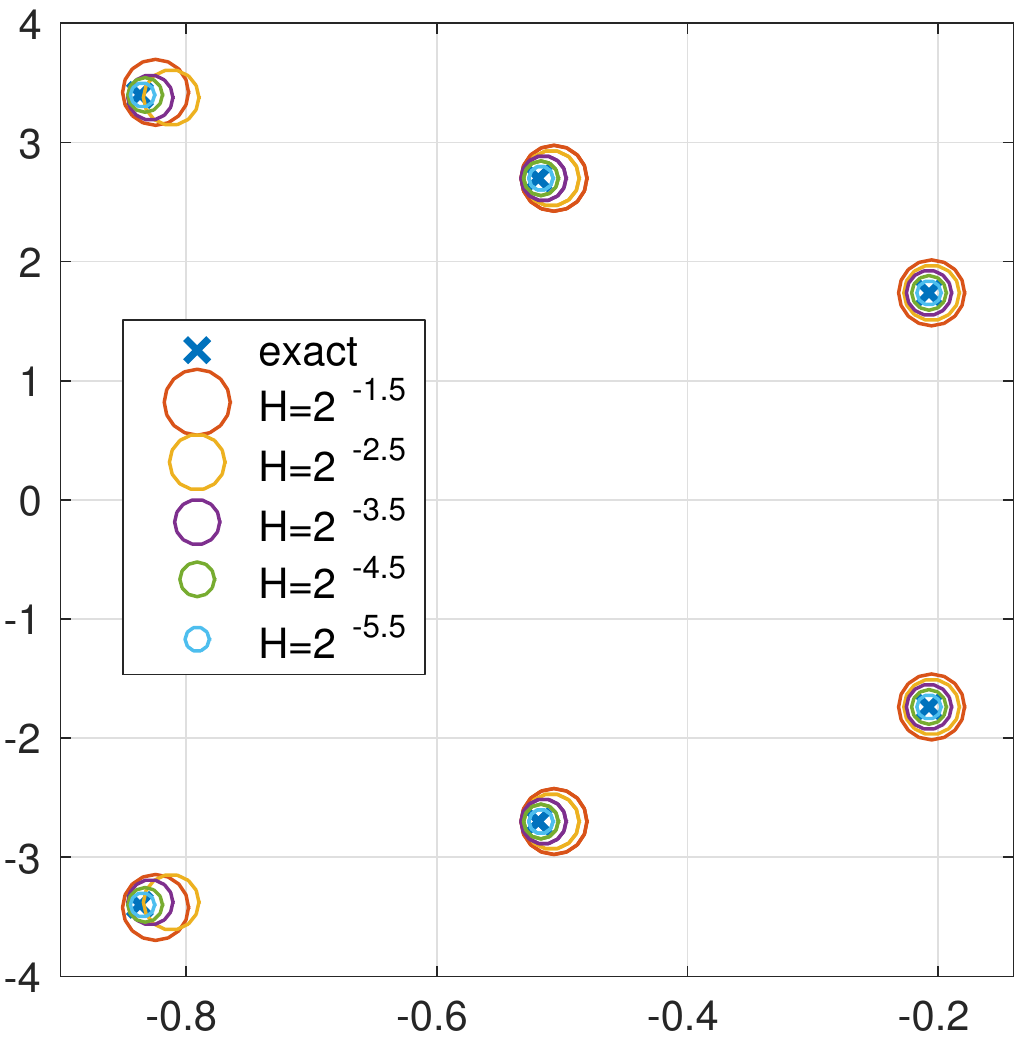}
\caption{Numerical experiment in Section~\ref{ss:numexpc} for discontinuous diffusion and discontinuous mass-type damping depicted in Fig.~\ref{fig:numexpbdata} (left and right). Left: Relative eigenvalue errors vs. coarse mesh size $H$ with localization parameter $\ell = \lceil 2\log(H^{-1})\rceil$. Right: Illustration of eigenvalue convergence as $H$ decreases. Note that the eigenvalues $\lambda_3,\lambda_5$ (resp. $\lambda_4,\lambda_6$) are clustered and can hardly be distinguished in the plot.}\label{fig:numexp4}
  \end{center}
\end{figure}
Now we consider stiffness type damping,
$$d(v,w)=(\tilde{\kappa}\nabla v,\nabla w).$$
We let $\alpha=0.1$, $\mu=0.015$, and $\beta=0.006$. In Figure \ref{fig:numexp4}, we see the convergence of the error as the coarse mesh size decreases. The parameter $\gamma$ in Assumption (B) is equal to zero for discontinuous $\tilde{\kappa}$ so that the theory does not predict any rate of convergence. Still, we detect $\mathcal{O}(H^{4/3})$ convergence roughly with pre-asymptotic effects only on the coarsest grid. Again, this may be related to the regularity of the underlying PDE eigenvalue problem which we do not take advantage of in the analysis.

\section{Conclusion}\label{s:conclusion}
We present an efficient numerical procedure applicable for a class of quadratic eigenvalue problems, discretized using conforming P1 finite elements. The main idea is to construct a low dimensional subspace of the finite element space that capture crucial features of the problem. In particular we consider problems with rapid data variation, modeling for instance composite materials. The numerical results are very promising. High convergence rates without pre-asymptotic regime are detected. This indicates that already low dimensional GFEM spaces give sufficient accuracy in the eigenvalue approximation. Therefore, the task of solving a very large quadratic eigenvalue problem can be replaced by the task of solving of many localized independent linear Poisson type problems followed by one small quadratic eigenvalue problem. This conclusion was also made for linear eigenvalue problems in \cite{MaPe13} and eigenvalue problems involving a non-linearity in the eigenfunction \cite{HeMaPe13}.

In the error analysis, we prove convergence with rate for the eigenvalues and eigenspaces. The result are based on classical works for non-symmetric eigenvalue problems nicely summarized in \cite{BaOs91} and recent work of the authors for linear eigenvalue problems. The results rely on qualitative assumptions on the size of $H$ and the distribution of the discrete and exact eigenvalues. Without explicit knowledge on the structure of the damping, it seems to be difficult to avoid them or replace them with more quantitative assumptions. However we have not yet observed any problems in the numerical experiments. A challenge for future research would be to rigorously justify the numerical performance in the targeted pre-asymptotic regime, at least for practically relevant classes of damping.


\begin{thebibliography}{99}

\bibitem{BaOs91}
I.~Babu\v{s}ka and J.~Osborn, \emph{Eigenvalue problems. In Handbook of numerical analysis, Vol. II,}
Handb.~Numer.~Anal., II, 641--787. North-Holland, Amsterdam, 1991.

\bibitem{BaOs99}
I.~Babu\v{s}ka and J.~Osborn, \emph{Can a finite element method perform arbitrary badly?,} Math.~Comp., 69 (1999), 443--462.

\bibitem{BeDuRoSo00}
A.~Berm\'{u}dez, R.~G.~Dur\'{a}n, R.~Rodr\'{i}guez, and J.~Solomin, \emph{Finite element analysis of a quadratic eigenvalue problem arising in dissipative acoustics,} SIAM J.~Numer.~Anal., 38 (2000), 267--291.

\bibitem{BeHiMeSc13} T.~Betcke, N.~Higham, V.~Mehrmann, and C.~Schr\"oder, \emph{
NLEVP: A Collection of Nonlinear Eigenvalue Problems,} ACM Transactions on Mathematical Software, 39 (2013), 7:1--7:28.

\bibitem{Bo10}
D.~Boffi, \emph{Finite element approximation of eigenvalue problems,} Acta Numer., 19:1--120, 2010.

\bibitem{BrPe14}
D.~{Brown} and D.~{Peterseim}.
\newblock A multiscale method for porous microstructures.
\newblock {\em ArXiv e-prints}, November 2014.

\bibitem{CV99}
C.~Carstensen and R.~Verf\"urth, \emph{ Edge residuals dominate a posteriori error estimates for low order finite
element methods,} SIAM J.~Numer.~Anal.~36 (1999), 1571--1587.

\bibitem{Ci78}
P.~G.~Ciarlet, \emph{The finite element method for elliptic problems,} Classics in Applied Mathematics, 40, SIAM, Philadelphia, 2002.

\bibitem{DeNaRa78a}
J.~Descloux, N.~Nassif, and J.~Rappaz, \emph{On spectral approximation. Part 1: The problem of convergence,} RAIRO Anal.~Num\'{e}r., 12 (1978), 97--112.

\bibitem{DeNaRa78b}
J.~Descloux, N.~Nassif, and J.~Rappaz, \emph{On spectral approximation. Part 2: Error estimates for the Galerkin method,} RAIRO Anal.~Num\'{e}r., 12 (1978), 113--119.

\bibitem{GaPe15} D.~Gallistl and D.~Peterseim, \emph{Stable multiscale Petrov-Galerkin finite element method for high frequency acoustic scattering,} Comp.~Meth.~Appl.~Mech.~Eng., 295 (2015), 1--17.

\bibitem{HeMaPe13} P.~Henning, A.~M\aa lqvist, and D.~Peterseim, \emph{Two-level discretization techniques for ground state computations of Bose-Einstein condensates,} SIAM J.~Numer.~Anal.,52 (2014), 1525--1550.

\bibitem{HaMoPe15}
P. Henning, P. Morgenstern, and D. Peterseim.
\newblock Multiscale partition of unity.
\newblock In Michael Griebel and Marc~Alexander Schweitzer, editors, {\em
  Meshfree Methods for Partial Differential Equations VII}, volume 100 of {\em
  Lecture Notes in Computational Science and Engineering}, pages 185--204.
  Springer International Publishing, 2015.

\bibitem{HePe13}
P.~Henning and D.~Peterseim.
\newblock Oversampling for the multiscale finite element method.
\newblock {\em Multiscale Modeling \& Simulation}, 11(4):1149--1175, 2013.

\bibitem{HwLiMe03} T.-M.~Hwang, W.-W.~Lin, and V:~Mehrmann, \emph{Numerical solution of quadratic eigenvalue problems with structure-perserving methods,} SIAM J.~Sci.~Comp., 24 (2003), 1283--1302.

\bibitem{Ka76} T.~Kato, \emph{Perturbation theory for linear operators,} Springer-Verlag, New York, 1976.

\bibitem{MeVo04} V.~Mehrmann and H.~Voss, \emph{Nonlinear eigenvalue problems: a challenge for modern eigenvalue methods,} GAMM-Mitt., 27 (2004), 121--152.

\bibitem{MaPr15} A.~M\aa lqvist and A.~Persson, \emph{Multiscale techniques for parabolic equations,} arXiv/1504.08140v1

\bibitem{MaPe11} A.~M\aa lqvist and D.~Peterseim, \emph{Localization of elliptic multiscale problems,} Math.~Comp., 83 (2014),
2583--2603.

\bibitem{MaPe13} A.~M\aa lqvist and D.~Peterseim, \emph{Computation of eigenvalues by numerical upscaling,} Numer.~Math., 130 (2015), 337--361.

\bibitem{Os75} J.~E.~Osborn, \emph{Spectral approximation for compact operators,} Math.~Comp., 29 (1975), 712--725.

\bibitem{Pe14}
D.~{Peterseim}.
\newblock Eliminating the pollution effect in {H}elmholtz problems by local
  subscale correction.
\newblock {\em ArXiv e-prints}, 1411.1944, 2014.

\bibitem{Pe15} D.~Peterseim, \emph{Variational multiscale stabilization and the
exponential decay of fine-scale correctors,} ArXiv e-prints, May 2015.

\bibitem{StFi73} G.~Strang and G.~Fix, \emph{An analysis of the finite element method,}

\bibitem{TiMe01} F.~Tisseur  and K.~Meerbergen,
\emph{A Survey of the Quadratic Eigenvalue Problem,} SIAM Review, 43, (2001) 235--286.

\bibitem{We97} J.~Weidmann, \emph{Strong operator convergence and spectral theory of ordinary differential operators,} Univ.~Iagel.~Acta Math., 34  (1997), 153--163.

\end{thebibliography}
\end{document}